\documentclass[11pt]{amsart}

\usepackage{etoolbox}

\makeatletter
\let\old@tocline\@tocline
\let\section@tocline\@tocline
\newcommand{\subsection@dotsep}{4.5}
\newcommand{\subsubsection@dotsep}{4.5}
\patchcmd{\@tocline}
{\hfil}
{\nobreak
	\leaders\hbox{$\m@th
		\mkern \subsection@dotsep mu\hbox{.}\mkern \subsection@dotsep mu$}\hfill
	\nobreak}{}{}
\let\subsection@tocline\@tocline
\let\@tocline\old@tocline

\patchcmd{\@tocline}
{\hfil}
{\nobreak
	\leaders\hbox{$\m@th
		\mkern \subsubsection@dotsep mu\hbox{.}\mkern \subsubsection@dotsep mu$}\hfill
	\nobreak}{}{}
\let\subsubsection@tocline\@tocline
\let\@tocline\old@tocline

\let\old@l@subsection\l@subsection
\let\old@l@subsubsection\l@subsubsection

\def\@tocwriteb#1#2#3{%
	\begingroup
	\@xp\def\csname #2@tocline\endcsname##1##2##3##4##5##6{%
		\ifnum##1>\c@tocdepth
		\else \sbox\z@{##5\let\indentlabel\@tochangmeasure##6}\fi}%
	\csname l@#2\endcsname{#1{\csname#2name\endcsname}{\@secnumber}{}}%
	\endgroup
	\addcontentsline{toc}{#2}%
	{\protect#1{\csname#2name\endcsname}{\@secnumber}{#3}}}%

\newlength{\@tocsectionindent}
\newlength{\@tocsubsectionindent}
\newlength{\@tocsubsubsectionindent}
\newlength{\@tocsectionnumwidth}
\newlength{\@tocsubsectionnumwidth}
\newlength{\@tocsubsubsectionnumwidth}
\newcommand{\settocsectionnumwidth}[1]{\setlength{\@tocsectionnumwidth}{#1}}
\newcommand{\settocsubsectionnumwidth}[1]{\setlength{\@tocsubsectionnumwidth}{#1}}
\newcommand{\settocsubsubsectionnumwidth}[1]{\setlength{\@tocsubsubsectionnumwidth}{#1}}
\newcommand{\settocsectionindent}[1]{\setlength{\@tocsectionindent}{#1}}
\newcommand{\settocsubsectionindent}[1]{\setlength{\@tocsubsectionindent}{#1}}
\newcommand{\settocsubsubsectionindent}[1]{\setlength{\@tocsubsubsectionindent}{#1}}

\renewcommand{\l@section}{\section@tocline{1}{\@tocsectionvskip}{\@tocsectionindent}{}{\@tocsectionformat}}%
\renewcommand{\l@subsection}{\subsection@tocline{1}{\@tocsubsectionvskip}{\@tocsubsectionindent}{}{\@tocsubsectionformat}}%
\renewcommand{\l@subsubsection}{\subsubsection@tocline{1}{\@tocsubsubsectionvskip}{\@tocsubsubsectionindent}{}{\@tocsubsubsectionformat}}%
\newcommand{\@tocsectionformat}{}
\newcommand{\@tocsubsectionformat}{}
\newcommand{\@tocsubsubsectionformat}{}
\expandafter\def\csname toc@1format\endcsname{\@tocsectionformat}
\expandafter\def\csname toc@2format\endcsname{\@tocsubsectionformat}
\expandafter\def\csname toc@3format\endcsname{\@tocsubsubsectionformat}
\newcommand{\settocsectionformat}[1]{\renewcommand{\@tocsectionformat}{#1}}
\newcommand{\settocsubsectionformat}[1]{\renewcommand{\@tocsubsectionformat}{#1}}
\newcommand{\settocsubsubsectionformat}[1]{\renewcommand{\@tocsubsubsectionformat}{#1}}
\newlength{\@tocsectionvskip}
\newcommand{\settocsectionvskip}[1]{\setlength{\@tocsectionvskip}{#1}}
\newlength{\@tocsubsectionvskip}
\newcommand{\settocsubsectionvskip}[1]{\setlength{\@tocsubsectionvskip}{#1}}
\newlength{\@tocsubsubsectionvskip}
\newcommand{\settocsubsubsectionvskip}[1]{\setlength{\@tocsubsubsectionvskip}{#1}}

\patchcmd{\tocsection}{\indentlabel}{\makebox[\@tocsectionnumwidth][l]}{}{}
\patchcmd{\tocsubsection}{\indentlabel}{\makebox[\@tocsubsectionnumwidth][l]}{}{}
\patchcmd{\tocsubsubsection}{\indentlabel}{\makebox[\@tocsubsubsectionnumwidth][l]}{}{}

\newcommand{\@sectypepnumformat}{}
\renewcommand{\contentsline}[1]{%
	\expandafter\let\expandafter\@sectypepnumformat\csname @toc#1pnumformat\endcsname%
	\csname l@#1\endcsname}
\newcommand{\@tocsectionpnumformat}{}
\newcommand{\@tocsubsectionpnumformat}{}
\newcommand{\@tocsubsubsectionpnumformat}{}
\newcommand{\setsectionpnumformat}[1]{\renewcommand{\@tocsectionpnumformat}{#1}}
\newcommand{\setsubsectionpnumformat}[1]{\renewcommand{\@tocsubsectionpnumformat}{#1}}
\newcommand{\setsubsubsectionpnumformat}[1]{\renewcommand{\@tocsubsubsectionpnumformat}{#1}}
\renewcommand{\@tocpagenum}[1]{%
	\hfill {\mdseries\@sectypepnumformat #1}}

\let\oldappendix\appendix
\renewcommand{\appendix}{%
	\leavevmode\oldappendix%
	\addtocontents{toc}{%
		\protect\settowidth{\protect\@tocsectionnumwidth}{\protect\@tocsectionformat\sectionname\space}%
		\protect\addtolength{\protect\@tocsectionnumwidth}{2em}}%
}
\makeatother

\makeatletter
\settocsectionnumwidth{2em}
\settocsubsectionnumwidth{2.5em}
\settocsubsubsectionnumwidth{3em}
\settocsectionindent{1pc}%
\settocsubsectionindent{\dimexpr\@tocsectionindent+\@tocsectionnumwidth}%
\settocsubsubsectionindent{\dimexpr\@tocsubsectionindent+\@tocsubsectionnumwidth}%
\makeatother

\settocsectionvskip{10pt}
\settocsubsectionvskip{0pt}
\settocsubsubsectionvskip{0pt}

\settocsectionformat{\bfseries}
\settocsubsectionformat{\mdseries}
\settocsubsubsectionformat{\mdseries}
\setsectionpnumformat{\bfseries}
\setsubsectionpnumformat{\mdseries}
\setsubsubsectionpnumformat{\mdseries}

\let\oldtableofcontents\tableofcontents
\renewcommand{\tableofcontents}{%
	\vspace*{-\linespacing}
	\oldtableofcontents}

\usepackage[dvipsnames]{xcolor}
\usepackage[hyperfootnotes=false]{hyperref}
\usepackage{enumerate}
\usepackage{tikz}
\usepackage{tikz-cd}
\usepackage{hyperref}
\usepackage{amsthm,amsfonts,amsmath,amscd,amssymb}
\usepackage{xcolor,float}
\usepackage[all]{xy}
\usepackage{mathrsfs}
\usepackage{graphics,graphicx}
\usepackage{caption}
\usepackage{mathtools}

\let\oldmarginpar\marginpar
\renewcommand\marginpar[1]{\-\oldmarginpar[\raggedleft\footnotesize #1]%
	{\raggedright\footnotesize #1}}
\theoremstyle{plain}
\newtheorem{thm}{Theorem}[section]

\newtheorem{example}[thm]{Example}
\newtheorem{prop}[thm]{Proposition}

\newtheorem{cor}[thm]{Corollary}

\newtheorem{conj}[thm]{Conjecture}
\theoremstyle{definition}

\newtheorem{definition}[thm]{Definition}
\newtheorem{remark}[thm]{Remark}

\newtheorem{computation}[thm]{Computation}

\theoremstyle{remark}

\numberwithin{equation}{section}

\renewcommand{\S}{\mathbb{S}}
\renewcommand{\P}{\mathbb{P}}
\renewcommand{\L}{\mathbb{L}}

\newcommand{\D}{\mathbb{D}}
\newcommand{\N}{\mathbb{N}}
\newcommand{\Z}{\mathbb{Z}}
\newcommand{\Q}{\mathbb{Q}}
\newcommand{\R}{\mathbb{R}}
\newcommand{\C}{\mathbb{C}}
\newcommand{\SA}{\mathcal{A}}
\renewcommand{\SS}{\mathcal{S}}

\newcommand{\SB}{\mathcal{B}}

\newcommand{\A}{\mathscr{A}}
\newcommand{\La}{\Lambda}
\newcommand{\la}{\lambda}

\renewcommand{\a}{\alpha}

\newcommand{\dd}{\partial}

\newcommand{\sse}{\subseteq}
\newcommand{\lr}{\longrightarrow}

\newcommand{\Br}{\operatorname{Br}}

\newcommand{\ob}{\operatorname{ob}}

\newcommand{\Symp}{\operatorname{Symp}}

\newcommand{\GL}{\operatorname{GL}}

\newcommand{\Sh}{\operatorname{Sh}}
\newcommand{\Ob}{\operatorname{ob}}
\newcommand{\Aug}{\operatorname{Aug}}
\newcommand{\Spec}{\operatorname{Spec}}

\newcommand{\mush}{\operatorname{\mu sh}}
\newcommand{\wt}{\tilde}

\newcommand{\st}{\text{st}}

\def\Op{{\mathcal O}{\it p}}
\newcounter{daggerfootnote}

\newcommand{\bD}{\mathbb{D}}

\newcommand{\bL}{\mathbb{L}}

\definecolor{arsenic}{rgb}{0.23, 0.27, 0.29}
\definecolor{aurometalsaurus}{rgb}{0.43, 0.5, 0.5}
\definecolor{battleshipgrey}{rgb}{0.52, 0.52, 0.51}
\definecolor{arylideyellow}{rgb}{0.91, 0.84, 0.42}
\definecolor{coolblack}{rgb}{0.0, 0.18, 0.39}

\DeclarePairedDelimiter\floor{\lfloor}{\rfloor}

\begin{document}
	
	\title{Lagrangian Skeleta and Plane Curve Singularities\vspace{0.25cm}}
	
	\author{Roger Casals\vspace{0.25cm}}
	\address{University of California Davis, Dept. of Mathematics, Shields Avenue, Davis, CA 95616, USA}
	\email{casals@math.ucdavis.edu}

\begin{abstract} We construct closed arboreal Lagrangian skeleta associated to links of isolated plane curve singularities. This yields closed Lagrangian skeleta for Weinstein pairs $(\C^2,\La)$ and Weinstein 4-manifolds $W(\La)$ associated to max-tb Legendrian representatives of algebraic links $\La\sse(\S^3,\xi_\st)$. We provide computations of Legendrian and Weinstein invariants, and discuss the contact topological nature of the Fomin-Pylyavskyy-Shustin-Thurston cluster algebra associated to a singularity. Finally, we present a conjectural ADE-classification for Lagrangian fillings of certain Legendrian links and list some related problems.
\end{abstract}
\thispagestyle{empty}

\maketitle


\section{Introduction}\label{sec:intro}

The object of this note is to study a relation between the theory of isolated plane curve singularities\footnote{The reader is referred to \cite{Ghys17} for a beautiful and gentle introduction to the subject.}, as developed by V.I. Arnol'd and S. Gusein-Zade \cite{Arnold90,AGZV1,AGZV2,GuseinZade74}, N. A'Campo \cite{ACampo73,ACampo75,ACampo98,ACampo99}, J.W. Milnor \cite{Milnor68} and others, and arboreal Lagrangian skeleta of Weinstein 4-manifolds. In particular, we construct {\it closed} Lagrangian skeleta for the infinite class of Weinstein 4-manifolds obtained by attaching Weinstein 2-handles \cite{CieliebakEliashberg12,Weinstein91} to the link of an arbitrary isolated plane curve singularity $f:\C^2\lr\C$. These closed Lagrangian skeleta allow for an explicit computation of the moduli of microlocal sheaves \cite{GKS_Quantization,NadlerShende20,STWZ} and also explain the symplectic topology origin of the Fomin-Pylyavskyy-Shustin-Thurston cluster algebra \cite{Morsification_FPST} of an isolated singularity.
\subsection{Main Results}

The advent of Lagrangian skeleta and sheaf invariants have underscored the relevance of Legendrian knots in the study of symplectic 4-manifolds \cite{CasalsMurphy19,CieliebakEliashberg12,GPS3,STWZ,STZ}. The theory of arboreal singularities, as developed by D. Nadler \cite{Nadler15,Nadler17}, provides a local-to-global method for the computation of categories of microlocal sheaves \cite{NadlerShende20}. These invariants, in turn, yield results in terms of Fukaya categories \cite{GPS3,GPS1}. The existence of arboreal Lagrangian skeleta has been crystallized by L. Starkston \cite{Starkston} in the context of Weinstein 4-manifolds, where this article takes place.

Given a Weinstein 4-manifold $(W,\la_\st)$, it is presently a challenge to describe an associated arboreal Lagrangian skeleta $\L\sse W$. In particular, there is no general method for finding {\it closed} arboreal Lagrangian skeleta\footnote{That is, a compact arboreal Lagrangian skeleta $\L\sse(W\,\la)$ such that $\dd\L=0$.}, or deciding whether these exist. This manuscript explores this question by introducing a new type of closed arboreal Lagrangian skeleta for Legendrian links $\La\sse(\S^3,\xi_\st)$ which are maximal-tb representatives of the {\it link} of an isolated plane curve singularity $f\in\C[x,y]$. The discussion in this note unravels thanks to a geometric fact:

\begin{thm}\label{thm:main}
Let $f\in\C[x,y]$ be an isolated plane curve singularity and $\La_f\sse(\S^3,\xi_\st)$ its associated Legendrian link. The Weinstein pair $(\C^2,\La_f)$ admits the closed arboreal Lagrangian skeleton $\bL(\tilde{f})=M_{\tilde f}\cup\vartheta$, obtained by attaching the Lagrangian $\D^2$-thimbles $\vartheta$ of $\tilde{f}$ to the Milnor fiber $M_{\tilde f}$, for any real Morsification $\tilde{f}\in\R[x,y]$.\hfill$\Box$ 
\end{thm}

The two objects $\La_f$ and $\bL(\tilde{f})$ in the statement of Theorem \ref{thm:main} require an explanation, which will be given. We rigorously define the notion of a {\it Legendrian} link $\La_f\sse(\S^3,\xi_\st)$ associated to the germ $f\in\C[x,y]$ of an isolated curve singularity in Section \ref{sec:main}. Note that the smooth link of the singularity $f\in\C[x,y]$, as defined by J. Milnor \cite{Milnor68}, and canonically associated to $f$, is naturally a {\it transverse} link $T_f\sse(\S^3,\xi_\st)$ \cite{Etnyre05,Geiges08,Giroux02}. The Legendrian link $\La_f\sse(\S^3,\xi_\st)$ will be a maximal-tb Legendrian approximation of $T_f$. The notation $(\C^2,\La_f)$ refers to the Weinstein pair $(\C^2,\mathscr{R}(\La_f))$, where $\mathscr{R}(\La_f)\sse(\S^3,\xi_\st)$ is a small (Weinstein) annular ribbon for the Legendrian link $\La_f$.

The Lagrangian skeleton $\bL(\tilde{f})$ is also defined in Section \ref{sec:main}. Note that the Milnor fibration of $f\in\C[x,y]$ is a {\it symplectic} fibration on $(\C^2,\omega_\st)$, whose symplectic fibers bound the transverse link $T_f\sse(\S^3,\xi_\st)$. Nevertheless, the Lagrangian skeleton $\bL(\tilde{f})$ is built from the underlying {\it topological} Milnor fiber and the vanishing cycles associated to a real Morsification. Indeed, $\bL(\tilde{f})$ is obtained by attaching the Lagrangian thimbles of the morsification $\tilde{f}$ to the (topological) Milnor fiber, which is Lagrangian in $\L(\wt f)$. Theorem \ref{thm:main} is a {\it relative} statement, being about a Weinstein {\it pair} $(\C^2,\La_f)$ and not just about a Weinstein manifold. Hence, it is useful in the {\it absolute} context, as follows.

Consider a Legendrian knot $\La\sse(\S^3,\xi_\st)$ in the standard contact 3-sphere and the Weinstein 4-manifold $W(\La)=\bD^4\cup_\La T^*\bD^2$ obtained by performing a 2-handle attachment along $\La$. A front projection for $\La$ (almost) provides an arboreal skeleton for the Weinstein 4-manifold $W(\La)$ \cite{Starkston}. Nevertheless, the computation of microlocal sheaf invariants from this model is far from immediate, nor exhibits the cluster nature of the moduli space of Lagrangian fillings. The symplectic topology of a Weinstein manifold is much more visible, and invariants more readily computed, from a {\it closed} arboreal Lagrangian skeleton, i.e. an arboreal Lagrangian skeleton which is compact and without boundary. In particular, Theorem \ref{thm:main} provides such a closed Lagrangian skeleton associated to a real Morsification:

\begin{cor}\label{cor:main} Let $f\in\C[x,y]$ be an isolated curve singularity and $\La_f$ its associated Legendrian link. The 4-dimensional Weinstein manifold $W(\La_f)=\bD^4\cup_{\La_f} (T^*\bD^2\cup\stackrel{\pi_0(\La_f)}{\ldots}\cup T^*\bD^2))$ admits the closed arboreal Lagrangian skeleton ${\bL}(\tilde{f})\cup_\dd (\bD^2\cup\stackrel{\pi_0(\La_f)}{\ldots}\cup\bD^2)$, obtained by attaching the Lagrangian $\D^2$-thimbles of $\tilde{f}$ to the compactified Milnor fiber $\overline{M}_f=M_f\cup_\dd(\bD^2\cup\stackrel{\pi_0(\dd M_f)}{\ldots}\cup\bD^2)$, for any real Morsification $\tilde{f}\in\R[x,y]$.\hfill$\Box$ 
\end{cor}

Let us see how Theorem \ref{thm:main} and Corollary \ref{cor:main} can be applied for two simple singularities, corresponding to the $D_5$ and the $E_6$ Dynkin diagrams. As we will see, part of the strength of these results is the explicit nature of the resulting Lagrangian skeleta and the direct bridge they establish between the theory of singularities and symplectic topology.

\begin{center}
	\begin{figure}[h!]
		\centering
		\includegraphics[scale=0.7]{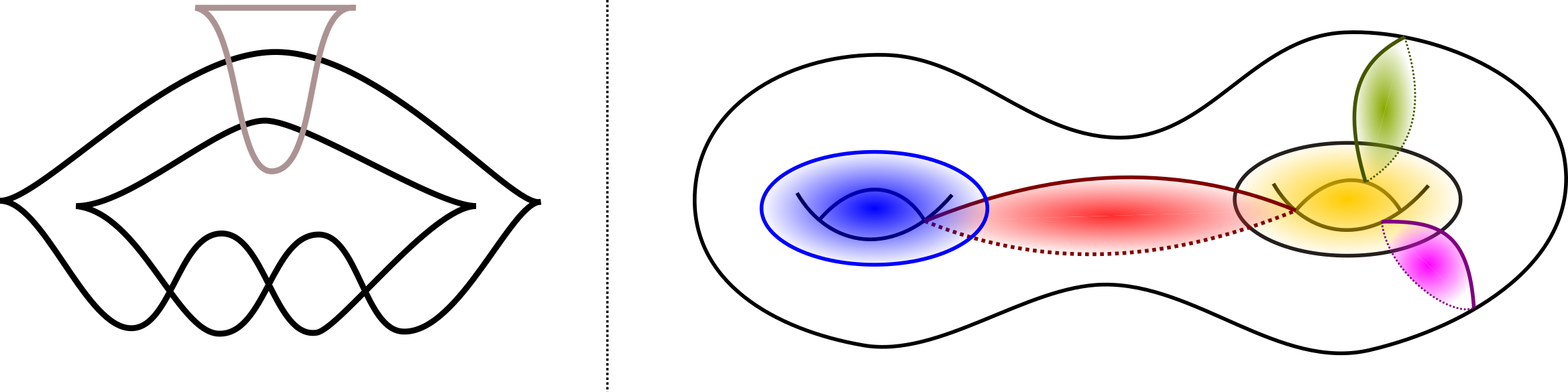}
		\caption{The $D_5$-Legendrian link $\La_f\sse(\S^3,\xi_\st)$ (Left) and a closed Lagrangian arboreal skeleton for the Weinstein 4-manifold $W(\La_f)$ (Right), obtained by attaching $5$ Lagrangian 2-disks to the cotangent bundle $(T^*\Sigma_2,\la_\st)$.}
		\label{fig:D5Example}
	\end{figure}
\end{center}

\begin{example}\label{ex:intro1} 
$($i$)$ First, consider the germ of the $D_5$-singularity $f(x,y)=xy^2+x^4$, the Legendrian link associated to this singularity is depicted in Figure \ref{fig:D5Example} (Left). The Weinstein 4-manifold $W(\La_f)=\bD^4\cup_{\La_f} (T^*\bD^2\cup T^*\bD^2)$ admits the closed arboreal Lagrangian skeleton depicted in Figure \ref{fig:D5Example} (Right). The $D_5$-Dynkin diagram is readily seen in the intersection quiver of the boundaries of the Lagrangian 2-disks added to the (smooth compactification) of the genus 2 Milnor fiber.

\begin{center}
	\begin{figure}[h!]
		\centering
		\includegraphics[scale=0.7]{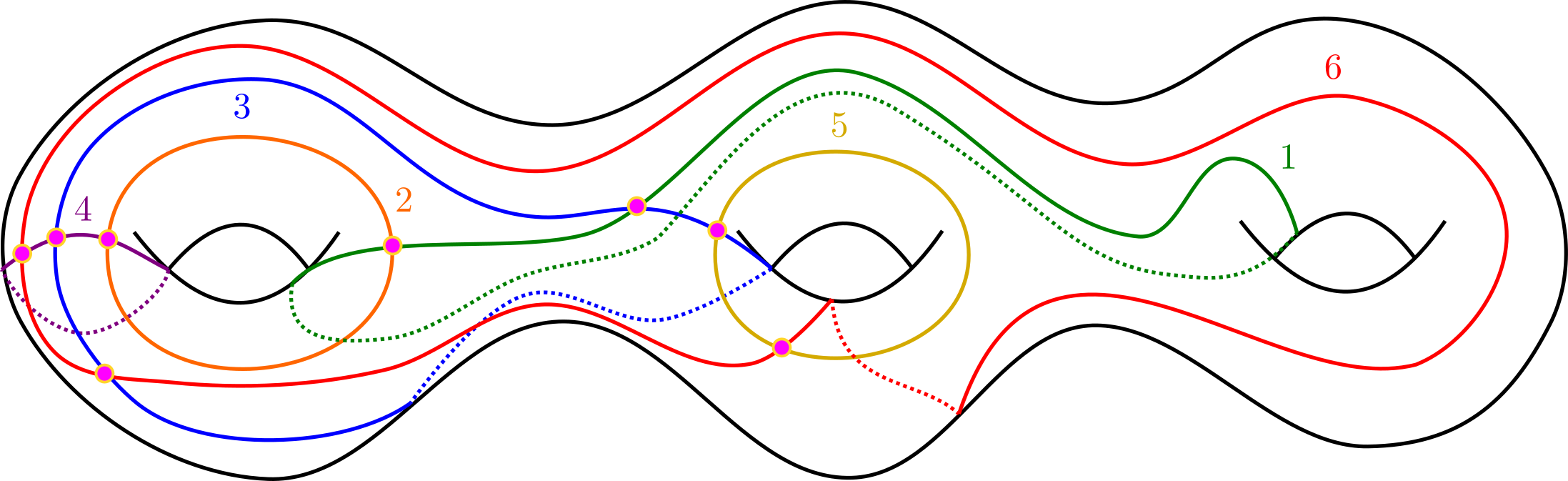}
		\caption{Closed Lagrangian arboreal skeleton associated to the $E_6$-simple singularity $f(x,y)=x^3+y^4$, according to Corollary \ref{cor:main}.}
		\label{fig:ExampleIntro}
	\end{figure}
\end{center}
	
$($ii$)$ Second, consider the germ of the singularity $f(x,y)=x^3+y^4$, the link of the singularity is the maximal-tb positive torus knot $\La_f\cong\La(3,4)\sse(\S^3,\xi_\st)$. The Weinstein 4-manifold $W(\La_f)=\bD^4\cup_{\La_f} T^*\bD^2$ admits the closed arboreal Lagrangian skeleton depicted in Figure \ref{fig:ExampleIntro}. This Lagrangian skeleton is built by attaching six Lagrangian 2-disks to the cotangent bundle $(T^*\Sigma_3,\la_\st)$ of a genus 3 surface. These 2-disks are attached along the six curves in Figure \ref{fig:ExampleIntro}, whose intersection pattern is $($mutation equivalent to$)$ the $E_6$ Dynkin diagram.\hfill$\Box$
\end{example}

From now onward, we abbreviate ``closed arboreal Lagrangian skeleton'' to {\it Cal}-skeleton.\footnote{This seems appropriate, as D. Nadler (UC Berkeley), L. Starkston (UC Davis) and Y. Eliashberg (Stanford), the initial developers of arboreal Lagrangian skeleta, hold their positions in the state of California.} Let $(W,\la)$ be a Weinstein 4-manifold, e.g. described by a Legendrian handlebody, a Lefschetz fibration or analytic equations in $\C^n$. There are two basic nested questions: Does it admit a Cal-skeleton ? If so, how do you find one ? For instance, consider a max-tb Legendrian representative $\La\sse(\dd\D^4,\la_\st)$ of any smooth knot, does $W(\La)$ admit a Cal-skeleton ? It might be that not all these Weinstein 4-manifolds $W(\La)$ admit such a skeleton: it is certainly not the case if the Legendrian knot $\La$ were stabilized, hence the max-tb hypothesis. In general, the lack of exact Lagrangians in $W(\La)$ would provide an obstruction.

\begin{remark}
For simplicity, we focus on {\it oriented} exact Lagrangians. Non-orientable Cal-skeleta should also be of interest. For instance, consider the max-tb Legendrian {\it left}-handed trefoil knot $\La(\overline{3}_1)\sse(\dd\D^4,\la_\st)$. Then $W(\La(\overline{3}_1))$ admits a Cal-skeleton $\R\P^2\cup_{\S^1}\D^2$ given by attaching a Lagrangian 2-disk to a Lagrangian $\R\P^2$, as shown in Figure \ref{fig:CalSkel_LeftHanded}.\hfill$\Box$
\end{remark}

\begin{center}
	\begin{figure}[h!]
		\centering
		\includegraphics[scale=0.8]{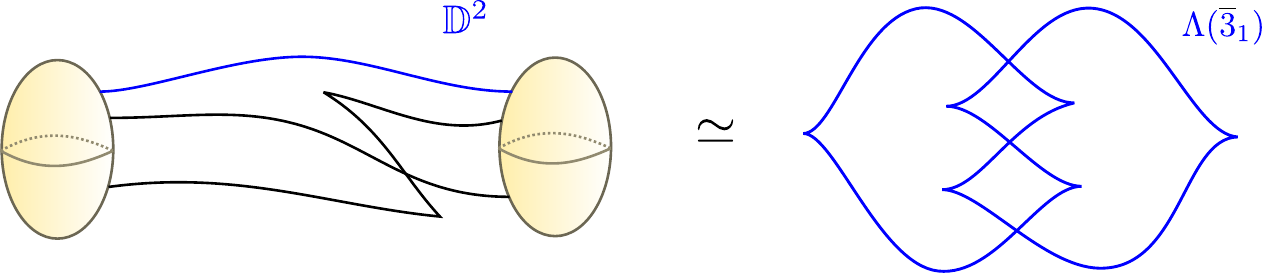}
		\caption{Cal-skeleton $\R\P^2\cup_{\S^1}\D^2$ associated to $\La(\overline{3}_1)\sse(\dd\D^4,\la_\st)$.}
		\label{fig:CalSkel_LeftHanded}
	\end{figure}
\end{center}

Symplectic invariants of Weinstein 4-manifolds $W$ include (partially) wrapped Fukaya categories \cite{Auroux14,Sylvan19_PartiallyWrapped} and categories of microlocal sheaves \cite{NadlerShende20}. Microlocal sheaf invariants should be particularly computable if a Cal-skeleton $\L\sse W$ is given, yet worked out examples are scarce in the literature. In Section \ref{sec:computations}, we use\footnote{The correspondence \cite[Theorem 1.3]{NRSSZ} and T. K\'alm\'an's description \cite{Kalman} of augmentation varieties $\Aug(\La)$ are also useful tools in this context.} Theorem \ref{thm:main} to compute the moduli space of simple microlocal sheaves on some of the Cal-skeleta $\L$ from Corollary \ref{cor:main}.

Finally, Theorem \ref{thm:main} provides a context for the study of exact Lagrangian fillings of Legendrian links $\La_f\sse(\S^3,\xi_\st)$ associated to isolated plane curve singularities. Indeed, let $\bL(f)=M_f\cup\vartheta$ be a Cal-skeleton for the Weinstein pair $(\C^2,\La_f)$, as produced in Theorem \ref{thm:main}. The topological Milnor fiber $M_f$ may serve as a marked exact Lagrangian filling for the Legendrian link $\La_f$, and performing Lagrangian disk surgeries \cite{STW_Combinatorics,Yau17_Surgery} along the Lagrangian thimbles in $\vartheta$ is a method to construct additional\footnote{Potentially not Hamiltonian isotopic.} exact Lagrangian fillings. In general, this strategy might be potentially obstructed, as the Lagrangian disks might acquire immersed boundaries when the Lagrangian surgeries are performed. That said, since Lagrangian disks surgeries yield combinatorial mutations of a quiver, Theorem \ref{thm:main} might hint towards a structural conjecture: we expect as many exact Lagrangian fillings $\La_f$ as elements in the cluster mutation class of the intersection quiver for the vanishing thimbles $\vartheta$. Section \ref{sec:conj} concludes with a discussion on such conjectural matters.\\

{\bf Acknowledgements}: The author thanks A. Keating for many conversations on divides of singularities throughout the years. The author is supported by the NSF grant DMS-1841913, a BBVA Research Fellowship and the Alfred P. Sloan Foundation.
\section{Lagrangian Skeleta for Isolated Singularities}\label{sec:main}

In this section we introduce the necessary ingredients for Theorem \ref{thm:main} and prove it. We refer the reader to \cite{AGZV1,Ghys17,Milnor65} for the basics of plane curve singularities and \cite{Etnyre03,Etnyre05,Geiges08,OzbagciStipsicz04} for background on 3-dimensional contact topology.

\subsection{The Legendrian Link of an Isolated Singularity}\label{ssec:LegOfSing}

Let $f\in\C[x,y]$ be a bivariate complex polynomial with an isolated complex singularity at the origin $(x,y)=(0,0)\in\C^2$. The {\it link of the singularity} $T_f\sse(\S^3,\xi_\st)$ is the intersection
$$T_f=V(f)\cap \S^3_\varepsilon=\{(x,y)\in\C^2:f(x,y)=0\}\cap\{(x,y)\in\C^2:|x|^2+|y|^2=\varepsilon\},$$
where $\varepsilon\in\R^+$ is small enough. The intersection is transverse for $\varepsilon\in\R^+$ small enough \cite{EisenbudNeumann85,Milnor68}, and thus $T_f$ is a smooth link. The link $T_f$ is in fact a transverse link for the contact structure $\xi_\st=T\S^3\cap i(T\S^3)$, as is the boundary of the (Milnor) fiber $M_f$ for the Milnor fibration \cite{Geiges08,Giroux02}. Equivalently, it is the transverse binding of the contact open book generated by
$$\frac{f}{\|f\|}:\S^3\setminus T_f\lr \S^1.$$
The link of a singularity was first introduced by W. Wirtinger and K. Brauner \cite{Brauner28} and masterfully studied by J. Milnor \cite{Milnor68}. The book \cite{EisenbudNeumann85} comprehensively develops\footnote{See also W. Neumann's article in E. K\"ahler's volume \cite{KahlerNeumann03}.} the smooth topology of link of singularities and their connection to 3-manifold topology. The contact topological nature of the associated open book was developed by E. Giroux \cite{Giroux02}.

From a smooth perspective, the smooth isotopy class of $T_f$ is that of an iterated cable of the unknot \cite{EisenbudNeumann85}. Let $K_{l,m}$ be the oriented $(l,m)$-cable of a smooth link $K\sse\S^3$, i.e. an embedded curve in the boundary $\dd\Op(K)$ of the solid torus $\Op(K)$ in the homology class $l\cdot[\la]+m\cdot[\mu]$, with $\la$ the longitude and $\mu$ the meridian of $\Op(K)$. It is shown in \cite[Chapter IV.7]{EisenbudNeumann85} that an iterated cable $K_{(l_1,m_1),(l_2,m_2),\ldots,(l_r,m_r)}\sse\S^3$ is the link of an isolated singularity if and only if $m_{i+1}>(l_im_i)l_{i+1}$, for $1\leq i\leq r-1$.

\begin{remark}\label{rmk:NewtonPuiseux} Given an isolated singularity $f(x,y)$, there are algorithms for determining the smooth type of $T_f$, i.e. the sequence of pairs $\{(l_1,m_1),(l_2,m_2),\ldots,(l_r,m_r)\}$. For instance, by applying the Newton-Puiseux algorithm to $f(x,y)$ we may write
$$y=a_1x^{\frac{n_1}{m_1}}+a_2x^{\frac{n_2}{m_1m_2}}+a_3x^{\frac{n_3}{m_1m_2m_3}}+\ldots,$$
at each branch; the pairs $(n_i,m_i)$ are called the Puiseux pairs. Then the cable pairs $(l_i,m_i)$ are given by $l_i=n_i-n_{i-1}m_i+m_{i-1}n_{i-1}n_i$. The algorithm is explained in \cite[Appendix to Chapter I]{EisenbudNeumann85}.\hfill$\Box$
\end{remark}

In the finer context of contact topology, the transverse link $T_f\sse(\S^3,\xi_\st)$ is an iterated cable with maximal self-linking number $sl(T_f)=\overline{sl}$, as it bounds the symplectic Milnor fiber $M_f\sse\C^2$ of $f\in\C[x,y]$, equiv.~the symplectic page of the contact open book \cite{Etnyre06,Giroux02}. By the transverse Bennequin bound \cite{Bennequin83}, this self-linking must be equal to the Euler characteristc $-\chi(M_f)$. A fact about the smooth isotopy class of links of singularities is their Legendrian simplicity:

\begin{prop}\label{prop:unique}
Let $f\in\C[x,y]$ be an isolated singularity and $T_f\sse(\S^3,\xi_\st)$. There exists a unique maximal Thurston-Bennequin Legendrian approximation $\La_f\sse(\S^3,\xi_\st)$ of the transverse link $T_f$.
\end{prop}

\begin{proof} The classification of Legendrian representatives of iterated cables of positive torus knots is established in \cite[Corollary 1.6]{LaFountain1}, building on \cite{EtnyreHonda05,EtnyreLafountainTosun12}. The sufficent numerical condition for Legendrian simplicity is $m_{i+1}/l_{i+1}>\overline{tb}(K_i)$, where $K_i$ is the $i$th iterated cable in $K_{(l_1,m_1),(l_2,m_2),\ldots,(l_r,m_r)}\sse\S^3$. The maximal Thurston-Bennequin equals $\overline{tb}(K_i)=A_r-B_r$, where $A_r,B_r\in\N$ are defined in \cite[Equation (2)]{LaFountain1}, and satisfy $m_il_i>A_i-B_i$. In particular, an algebraic link satisfies $m_{i+1}/l_{i+1}>m_il_i\geq\overline{tb}(K_i)$, for all $1\leq i\leq r-1$, and its max-tb representative is unique.
\end{proof}

Proposition \ref{prop:unique} implies that there exists a {\it unique} Legendrian link $\La_f\sse(\S^3,\xi_\st)$, up to contact isotopy, whose positive transverse push-off $\tau(\La_f)$, as defined in \cite[Section 3.5.3]{Geiges08}, is transverse isotopic to the transverse link $T_f$. Note that two distinct Legendrian approximations of a transverse link \cite[Theorem 2.1]{EpsteinFuchsMeyer01} differ by Legendrian stabilizations, which necessarily decrease the Thurston-Bennequin invariant.

\begin{remark}
Proposition \ref{prop:unique} does not hold for $K\sse(\S^3,\xi_\st)$ an arbitrary smooth link. For instance, the smooth isotopy classes of the mirrors $\overline{5}_2,\overline{6}_1$ of the three-twist knot and the Stevedore knot admit {\it two} distinct maximal-tb Legendrian representatives each \cite[Section 4]{Ng_Atlas}. That said, the knots $\overline{5}_2,\overline{6}_1$ are not links of singularities, as their Alexander polynomials are not monic, and thus they are not fibered knots \cite{Neuwirth75}.\hfill$\Box$
\end{remark}

Proposition \ref{prop:unique} allows us to canonically define a {\it Legendrian} link associated to an isolated singularity:

\begin{definition}\label{def:legendrian} A Legendrian link $\La_f\sse(\S^3,\xi_\st)$ is associated to an isolated singularity $f\in\C[x,y]$ if it is a maximal-tb Legendrian link $\La_f\sse(\S^3,\xi_\st)$ whose positive transverse push-off $\tau(\La_f)$ is transversely isotopic to the link of the singularity $T_f\sse(\S^3,\xi_\st)$.\hfill$\Box$
\end{definition}

Proposition \ref{prop:unique} shows that the Legendrian isotopy class of a Legendrian link $\La_f\sse(\S^3,\xi_\st)$ associated to an isolated singularity $f\in\C[x,y]$ is unique. Thus, we refer to $\La_f\sse(\S^3,\xi_\st)$ in Definition \ref{def:legendrian} as {\it the} Legendrian link associated to the isolated singularity $f\in\C[x,y]$.

\begin{example}[ADE Singularities]\label{ex:ADESingularities}
Let us consider the three ADE families of simple isolated singularities \cite[Chapter 2.5]{AGLV98}. Their germs are given by
$$(A_n)\quad f(x,y)=x^{n+1}+y^2,\qquad (D_n)\quad f(x,y)=xy^2+x^{n-1},\qquad n\in\N,$$
$$(E_6)\quad f(x,y)=x^3+y^4,\qquad(E_7)\quad f(x,y)=x^3+xy^3,\qquad(E_8)\quad f(x,y)=x^3+y^5.$$
\begin{center}
	\begin{figure}[h!]
		\centering
		\includegraphics[scale=0.8]{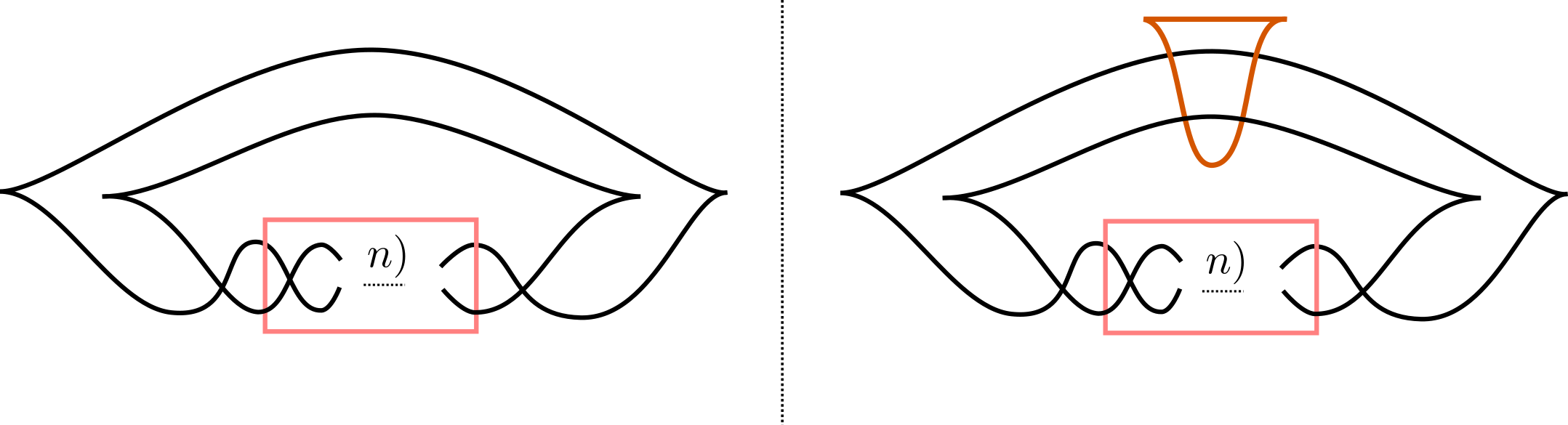}
		\caption{The Legendrian link for the $A_n$-singularity is the max-tb $(2,n+1)$-torus link (Left). The Legendrian link for the $D_n$-singularity is the link given by the union of a max-tb $(2,n-2)$-torus link and a standard Legendrian unknot, in orange, linked as in the Legendrian front on the right (Right).}
		\label{fig:ADSingularities}
	\end{figure}
\end{center}

The Legendrian link associated to the $A_n$-singularity is the positive $(2,n+1)$-torus link, with $\overline{tb}=n-1$. These links are associated to the braid $\sigma_1^{n+1}$, as depicted in Figure \ref{fig:ADSingularities} (Left). The Legendrian link associated to the $D_n$-singularity is the link consisting of the link associated to the $A_{n-3}$-singularity and the standard Legendrian unknot, linked as in Figure \ref{fig:ADSingularities} (Right). This is the topological consequence of the factorization $f(x,y)=x(y^2+x^{n-2})$. These $D_n$-links are associated to the (rainbow closure of the) positive braid $\sigma_1^{n-2}\sigma_2\sigma_1^2\sigma_2$, $n\geq3$. The $D_2$-link is the three-copy Reeb push-off of the Legendrian unknot, and the $D_3$-link is Legendrian isotopic to the $A_3$-link, i.e. a max-tb positive $T(2,4)$-torus link.

\begin{center}
	\begin{figure}[h!]
		\centering
		\includegraphics[scale=0.8]{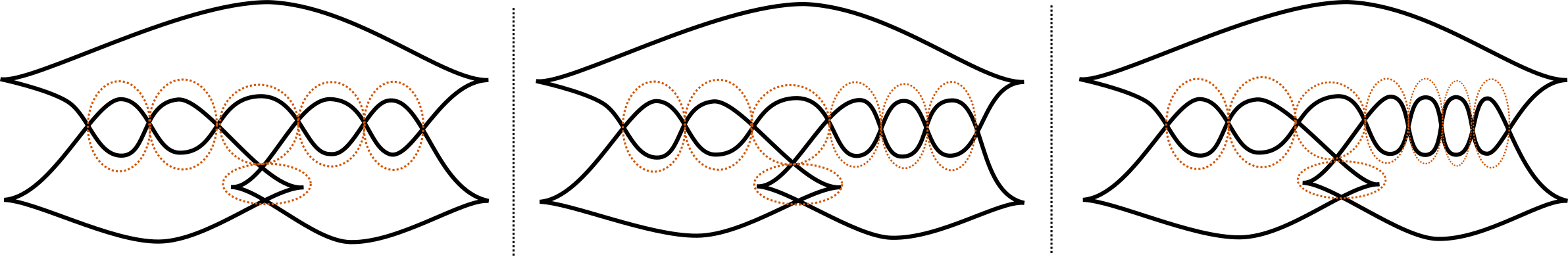}
		\caption{The Legendrian links for the $E_6,E_7$ and $E_8$ simple singularities.}
		\label{fig:ESingularities}
	\end{figure}
\end{center}

The Legendrian links associated to the $E_6$ and $E_8$ singularities are the maximal-tb positive $(3,4)$-torus Legendrian link and the Legendrian $(3,5)$-torus link, as depicted in Figure \ref{fig:ESingularities}. The $E_7$ is a maximal-tb Legendrian link consisting of a trefoil knot and a standard Legendrian unknot, linked as in the center Legendrian front in Figure \ref{fig:ESingularities}. This is implied by the $f(x,y)=x(x^2+y^3)$ factorization of the $E_7$ singularity. The Legendrian links for $E_6,E_7$ and $E_8$ can also be obtained as the closure of the three braids $\sigma_1^{n-3}\sigma_2\sigma_1^3\sigma_2$, $n=6,7,8$. Figure \ref{fig:ESingularities} also depicts generators of the first homology group of the minimal genus Seifert surface; these generate the first homology of each Milnor fiber, and the $E_6,E_7$ and $E_8$ Dynkin diagrams are readily exhibited from their intersection pattern.\hfill$\Box$
\end{example}

The singularities $f(x,y)=x^a+y^b$, $a\geq3,b\geq6$, or $(a,b)=(4,4),(4,5)$, yield an infinite family of non-simple isolated singularities for which the associated Legendrian is readily computed to be the maximal-tb positive $(a,b)$-torus link, confer Remark \ref{rmk:NewtonPuiseux}. Two more instances are illustrated in the following:

\begin{example} $($Two Iterated Cables$)$ Consider the isolated curve singularity
$$g(x,y)=x^7-x^6+4x^5y+2x^3y^2-y^4.$$
The Puiseux expansion yields the Newton solution $y=x^{3/2}(1+x^{1/4})$ and thus $\La_f\sse(\S^3,\xi_\st)$ is the maximal-tb Legendrian representative of the $(2,13)$-cable of the trefoil knot. This Legendrian knot is depicted in Figure \ref{fig:Cable23213} $($Left$)$. The reader is invited to show that the Legendrian knot $\La_f\sse(\S^3,\xi_\st)$ of the singularity
$$h(x,y)=x^9-x^{10}+6x^8y-3x^6y^2+2x^5y^3+3x^3y^4-y^6,$$
is the maximal-tb Legendrian representative of the $(3,19)$-cable of the trefoil knot \cite{Ghys17}, as depicted in Figure \ref{fig:Cable23213} $($Right$)$. $($For that, start by writing the relation as $y(x)=x^{3/2}+x^{5/3}$.$)$\hfill$\Box$
	\begin{center}
	\begin{figure}[h!]
		\centering
		\includegraphics[scale=0.8]{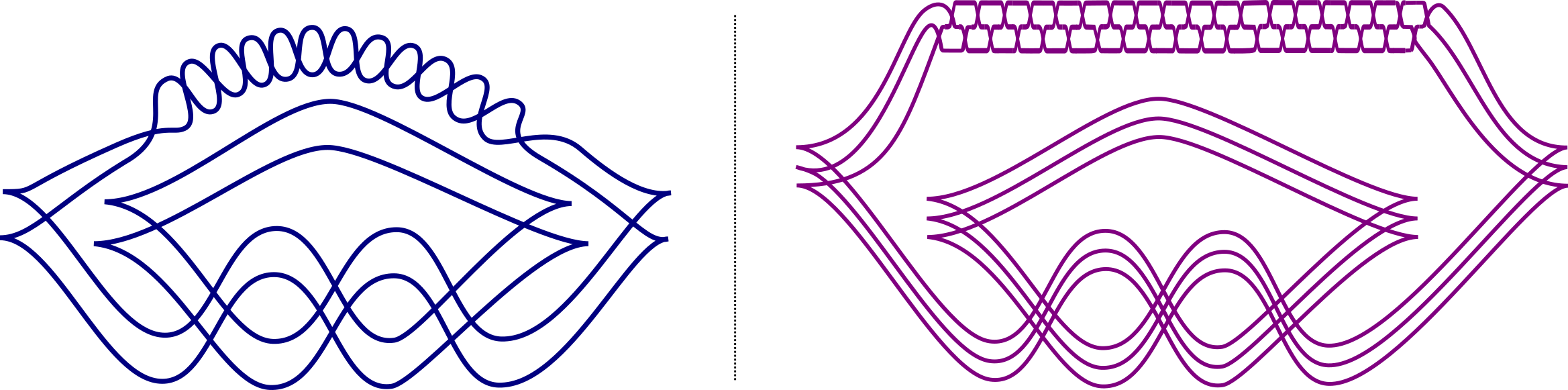}
		\caption{The Legendrian links $\La_g$ and $\La_h$ associated to the singularity $g(x,y)=x^7-x^6+4x^5y+2x^3y^2-y^4$, on the left, and the singularity $h(x,y)=x^9-x^{10}+6x^8y-3x^6y^2+2x^5y^3+3x^3y^4-y^6$, on the right.}
		\label{fig:Cable23213}
	\end{figure}
\end{center}
\end{example}

\subsection{A'Campo's Divides and Their Conormal Lifts}\label{ssec:ACampoDivide}

Let $f\in\C[x,y]$ be an isolated singularity, $\bD^4_\varepsilon\sse\C^2$ a Milnor ball for this singularity \cite[Corollary 4.5]{Milnor65}, $\varepsilon\in\R^+$, $\R^2=\{(x,y)\in\C^2:\Im(x)=0,\Im(y)=0\}\sse\C^2$ the real 2-plane and $\bD^2_\varepsilon=\bD^4_\varepsilon\cap\R^2$ a real Milnor 2-disk. Consider a real Morsification  $\tilde{f}_t(x,y)$, $t\in[0,1]$, such that, for $t\in(0,1]$, $f_t(x,y)$ has only $A_1$-singularities, its critical values are real and the level set $f^{-1}_t(0)\cap\bD^4_\varepsilon$, contains all the saddle points of the restriction $(f_t)|_{\bD^2_\varepsilon}$. The intersection $D_f=f^{-1}_t(0)\cap\bD^2_\varepsilon\sse\R^2$, where $\tilde{f}=f_1$,  is known as the {\it divide} of the real Morsification $\tilde{f}$ \cite{ACampo98,AGZV1,Ishikawa04}. It is the image of a union $I$ of closed segments under an immersion $i:I\lr\R^2$ \cite{GibsonIshikawa02,Hirasawa02,IshikawaNaoe20}, and we assume it is a generic such immersion. By considering $I\sse\R^2$ as a wavefront, its biconormal lift \cite{Arnold90} is a Legendrian link $\La_0(D_f)$ in the contact boundary $(\dd(T^*\R^2),\la_\st|_{\dd(T^*\R^2)})$. See \cite{ACampo75,GuseinZade74} for the existence and details of real Morsifications.

The biconormal lift $\La_0(D_f)\sse \dd(T^*\R^2)$ of the immersed curve $D_f$ to the (unit) boundary of the cotangent bundle $T^*\R^2$ can be constructed using the three local models:

\begin{itemize}
	\item[(i)] The biconormal lift near a smooth interior point $P\in D_f$ is defined as
	$$\{u\in T^*\Op(P):\|u_q\|=1,T_qD_f\sse\ker(u_q)\mbox{ for }q\in D_f\cap\Op(P)\},$$
	for an arbitrary fixed choice of metric in $\R^2$, and neighborhood $\Op(P)\sse\R^2$.\\
	\item[(ii)] The biconormal lift near an immersed point $P\in D_f$ is defined as the (disjoint) union of the conormal lifts of each of its embedded branches through $P$.\\
	\item[(iii)] Finally, at the endpoint $P\in D_f$, the biconormal lift is defined as the closure in $T^*_P\R^2$ of one of the components of
	$$T^*_P\R^2\setminus\{u\in T^*_P\R^2:\|u_q\|=1,T_PD_f\sse\ker(u_P)\mbox{ for }q\in D_f\cap\Op(P)\},$$
	where the tangent line $T_PD_f$ is defined as the (ambient) smooth limit of the tangent lines $T_{q_i}D_f$ for a sequence $\{q_i\}_{i\in\N}$ of interior points $q_i\in D_f$ convering to $P\in D_f$. There are two such components, but our arguments are independent of such a choice.
\end{itemize}   

\begin{remark}
The restriction of the canonical projection $\pi:\dd(T^*\R^2)\lr\R^2$ is finite two-to-one onto the image of the interior points of $I$. The pre-image of $\pi$ at (the image of) endpoints contains an open interval of the Legendrian circle fiber. For instance, the full conormal lift of a point $p\in\R^2$ is Legendrian isotopic to the zero section $\S^1\sse(J^1\S^1,\xi_\st)$, as is the conormal lift of an embedded closed segment.\hfill$\Box$
\end{remark}

These local models define the Legendrian biconormal lift $\La_0(D_f)\sse (\dd(T^*\R^2),\xi_\st)$ of the divide of the Morsification $\tilde{f}$. Let $\iota_0:\S^1\lr(\S^3,\xi_\st)$ be a Legendrian embedding in the isotopy class of the standard Legendrian unknot. A small neighborhood $\Op(\iota(\S^1))$ is contactomorphic to the $1$-jet space $(J^1\S^1,\xi_\st)\cong(T^*\S^1\times\R_t,\ker\{\la_\st-dt\})$, yielding a contact inclusion $\iota:(J^1\S^1,\xi_\st)\lr(\S^3,\xi_\st)$. Note that there exists a contactomorphism $\Psi:(\dd(T^*\R^2),\xi_\st)\lr(J^1\S^1,\xi_\st)$, where the zero section in the $1$-jet space bijects to the Legendrian boundary of a Lagrangian cotangent fiber in $T^*\R^2$. This leads to the following:

\begin{definition}\label{def:Leglift}
Let $D_f\sse\R^2$ be the divide associated to a real Morsification of an isolated singularity $f\in\C[x,y]$. The biconormal lift $\La(D_f)\sse (\S^3,\xi_\st)$ is the image $\iota(\Psi(\La_0(D_f)))$. That is, the biconormal lift $\La(D_f)\sse (\S^3,\xi_\st)$ is the satellite of the biconormal lift $\La_0(D_f)\sse (\dd(T^*\R^2),\xi_\st)$ with companion knot the standard Legendrian unknot in $(\S^3,\xi_\st)$.\hfill$\Box$
\end{definition}

The central result in N. A'Campo's articles \cite{ACampo98,ACampo99} is that the Legendrian link $\La(D_f)\sse S^3$ is {\it smoothly} isotopic to the transverse link $T_f$, see also \cite{IshikawaNaoe20}. The formulation above, in terms of the satellite to the Legendrian unknot, is not necessarily explicit in the literature on divides and their Legendrian lifts, but probably known to the experts, as it is effectively being used in M. Hirasawa's visualization \cite[Figure 2]{Hirasawa02}. See also the work of T. Kawamura \cite[Figure 2]{Kawamura02}, M. Ishikawa and W. Gibson \cite{GibsonIshikawa02,Ishikawa04} and others \cite{Chmutov03,IshikawaNaoe20}. The phrasing in Definition \ref{def:Leglift} might help crystallize the contact topological characteristics of each object.

\begin{center}
	\begin{figure}[h!]
		\centering
		\includegraphics[scale=0.7]{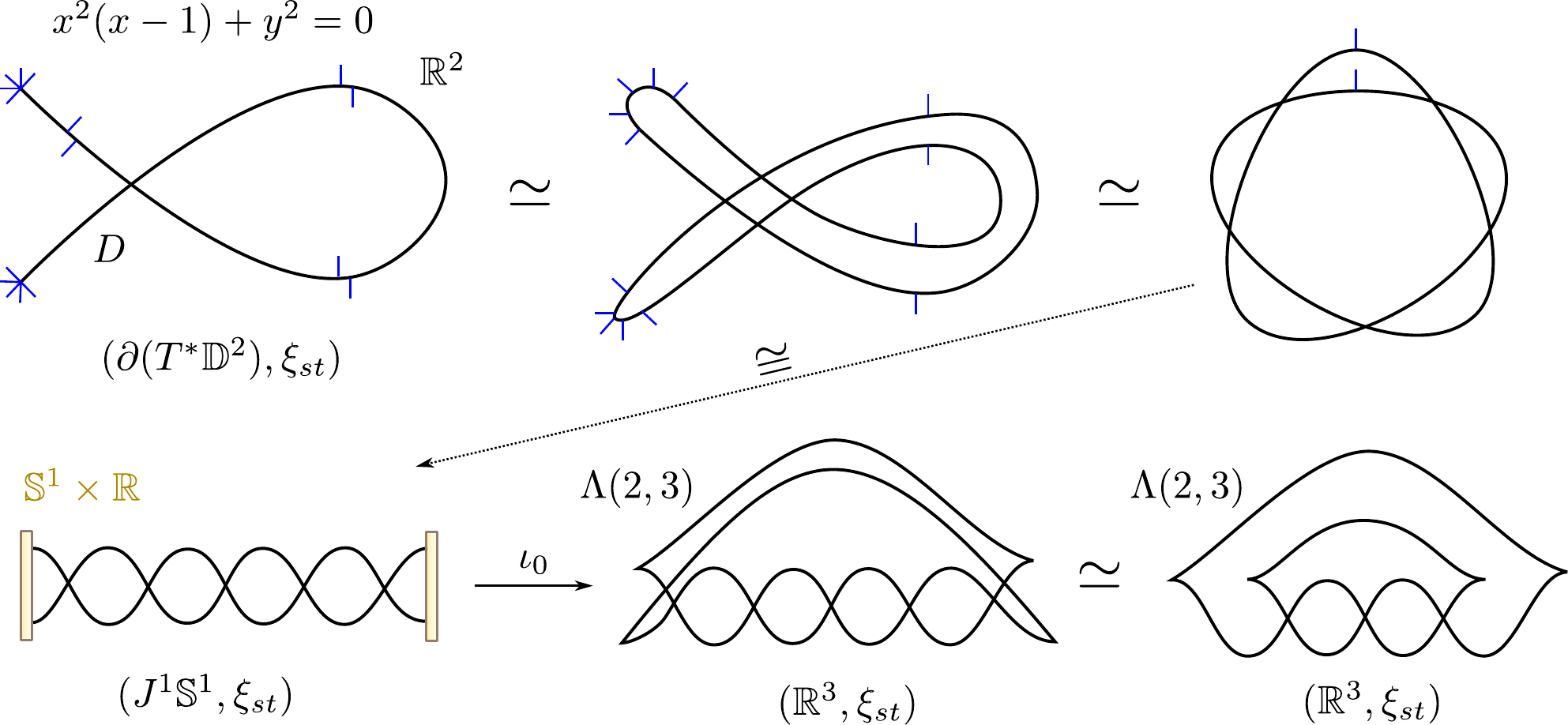}
		\caption{A co-oriented divide $D$ for the $A_2$-singularity $f(x,y)=x^3+y^2$, as a front for its Legendrian link $\La(D)\sse(\dd(T^*\D^2),\xi_\st)$. That is, the biconormal lift of $D$ is $\La(D)$. Its satellite along the standard unknot is the (unique) max-tb Legendrian trefoil $\La(2,3)\sse(\R^3,\xi_\st)$.}
		\label{fig:A2Singularity}
	\end{figure}
\end{center}

\begin{example}
$($i$)$ The $A_1$-singularity admits two real Morsifications $\tilde{f}_1(x,y)=x^2+y^2-1$ and $\tilde{f}_2(x,y)=x^2-y^2$, with corresponding divides
$$D_1=\{(x,y)\in\R^2:x^2+y^2-1=0\},\quad D_2=\{(x,y)\in\R^2:x^2-y^2=0\}.$$
The biconormal lift $\La_0(D_1)\sse(\dd T^*\R^2,\xi_\st)$ consists of two copies of the Legendrian fibers of the fibration $\pi:\dd T^*\R^2\lr\R^2$. Each of these two copies is satellited to the standard Legendrian unknot, forming a maximal-tb Hopf link $\La(D_1)\sse(\S^3,\xi_\st)$. Indeed, the second Legendrian fiber can be assumed to be the image of the first Legendrian fiber under the Reeb flow. Hence, the Legendrian link $\La(D_1)\sse(\S^3,\xi_\st)$ must consist of the standard Legendrian unknot union a small Reeb push-off. Similarly, the biconormal lift $\La_0(D_2)\sse(\dd T^*\R^2,\xi_\st)$ equally consists of two copies of the Legendrian fibers of the fibration $\pi:\dd T^*\R^2\lr\R^2$, and thus both Legendrian links $\La(D_1),\La(D_2)$ are Legendrian isotopic in $(\S^3,\xi_\st)$.\\

$($ii$)$ The $A_2$-singularity $f(x,y)=x^3+y^2$ admits the real Morsification $\tilde{f}(x,y)=x^2(x-1)+y^2$, whose divide is
$D=\{(x,y)\in\R^2:x^2(x-1)+y^2=0\}$. The divide $D\sse\R^2$ with its co-orientations is depicted in Figure \ref{fig:A2Singularity}. The first row depicts a wavefront homotopy, which yields a Legendrian isotopy in $(\dd T^*\R^2,\xi_\st)$. The second row starts by depicting the change of front projections induced by the contactomorphism $\Psi$, and performs the satellite to the standard Legendrian unknot. The resulting Legendrian $\La_f\sse(\S^3,\xi_\st)$ is the max-tb Legendrian trefoil knot $\La(2,3)$.

In general, divides for $A_n$-singularities are depicted in \cite[Figure 4]{Morsification_FPST}. We invite the reader to study the $A_5$-singularity $f(x,y)=x^5+y^2$ with its divide
$$D=\{(x,y)\in\R^2:x^2(x^3+x^2-x-1)+y^2=0\}$$
and discover the corresponding Legendrian isotopy, as in Figure \ref{fig:A2Singularity}. The isotopy should end with the max-tb Legendrian link $\La(2,5)\sse(\S^3,\xi_\st)$, e.g. expressed as the (rainbow) closure of the positive braid $\sigma_1^3$, equiv.~the $(-1)$-framed closure of $\sigma_1^5$. The general case $n\in\N$ is similar.\hfill$\Box$
\end{example}

\subsection{Proof of Theorem \ref{thm:main}}\label{ssec:proofmain}

There is an interesting dissonance at this stage. The Legendrian link $\La(D_f)\sse\S^3$ in Definition \ref{def:Leglift} and the transverse link $T_f\sse\S^3$ of the singularity are smoothly isotopic, yet certainly {\it not} contact isotopic. Their relationship is described by the following:

\begin{prop}\label{prop:TransPush}
Let $f\in\C[x,y]$ be an isolated singularity and $D_f\sse\R^2$ the divide associated to a real Morsification. The positive transverse push-off $\tau(\La(D_f))\sse(\S^3,\xi_\st)$ of the Legendrian link $\La(D_f)$ is contact isotopic to the transverse link $T_f\sse (\S^3,\xi_\st)$. In particular, $\La(D_f)\sse(\S^3,\xi_\st)$ is Legendrian isotopic to the Legendrian link $\La_f\sse(\S^3,\xi_\st)$ associated to the singularity $f\in\C[x,y]$.\hfill$\Box$
\end{prop}

\begin{proof} In A'Campo's isotopy \cite[Section 3]{ACampo98} from the link associated to the divide to the link of the singularity, the key step is the {\it almost complexification} of the Morsification $\wt{f}:\R^2\lr\R$. This replaces the $\R$-valued function $\wt{f}$ by an expression of the form
$$\wt{f}_\C:T^*\R^2\lr\C,\quad \wt{f}_\C(x,u):=\wt{f}(x)+i d\wt{f}(x)(u)-\frac{1}{2}\chi(x)H(f(x))(u,u),$$
which is a $\C$-valued function, where $u=(u_1,u_2)\in\R^2$ are Cartesian coordinates in the fiber. Here $H(f(x))$ is the Hessian of $f$, which is a quadratic form, and $\chi(x)$ is a bump function with $\chi(x)\equiv 1$ near double-points of the divide $D_f\sse\R^2$ and $\chi(x)\equiv 0$ away from them. The results in \cite{ACampo98}, see also \cite{IshikawaNaoe20,Ishikawa04}, imply that the transverse link of the singularity is isotopic to the intersection $T^\varepsilon\R^2\cap f^{-1}_\C(0)\sse (T^\varepsilon\R^2,\xi_\st)$ of the $\varepsilon$-unit cotangent bundle with the $0$-fiber of $\widetilde{f}_\C$, $\varepsilon\in\R^+$ small enough.\footnote{This mimicks S. Donaldson's construction of Lefschetz pencils, where the boundary of a fiber is a transverse link at the boundary, see also E. Giroux's construction of the contact binding of an open book \cite{Giroux17,Giroux02}.} It thus suffices to compare this transverse link to the Legendrian lift $\La(D_f)\sse (T^\varepsilon\R^2,\xi_\st)$, which we can check in each of the two local models: near a smooth interior point of the divide $D_f$ and near each of its double points. Note that the case of boundary points can be perturbed to that of smooth interior points, as in the first perturbation in Figure \ref{fig:A2Singularity}. We detail the computation in the first local model, the case of double points follows similarly.

The contact structure $(T^\varepsilon\R^2,\xi_\st)$ admits the contact form $\xi_\st=\ker\{\cos(\theta)dx_1-\sin(\theta)dx_2\}$, $(x_1,x_2)\in\R^2$ and $\theta\in\S^1$ is a coordinate in the fiber -- this is the angular coordinate in the $(u_1,u_2)$-coordinates above. The divide can be assumed to be cut locally by $D=\{(x_1,x_2)\in\R^2:x_2=0\}\sse\R^2$, as we can write $\wt{f}(x_1,x_2)=x_2$, and thus its bi-conormal Legendrian lift is
$$\La(D)=\{(x_1,x_2,\theta)\in\R^2\times\S^1:x_2=0,\theta=\pm\pi/2\}.$$
Note that the tangent space $T_{(x_1,x_2)}\La(D)$ of $\La(D)$ is spanned by $\dd_{x_1}$, which satisfies
$$\langle\dd_{x_1}\rangle=\ker\{\cos(\theta)dx_1-\sin(\theta)dx_2\},\mbox{ as }\cos(\theta)=0\mbox{ at }\theta=\pm\pi/2.$$

Since the model is away from a double point, $\wt{f}_\C(x,u):=x_2+i(0,1)\cdot(u_1,u_2)^t=x_2+iu_2$ becomes the standard symplectic projection $\R^2\times\R^2\lr\R^2$ onto the second (symplectic) factor. The zero set is thus $x_2=0$ and $u_2=0$ and so the intersection with $T^\varepsilon\R^2$ is
$$\kappa=\{(x_1,x_2,\theta)\in\R^2\times\S^1:x_2=0,\theta=0,\pi\},$$
as the points with $|u_1|^2=\varepsilon$ are at $\theta$-coordinates $\theta=0,\pi$. The tangent space $T\kappa=\langle \dd_{x_1}\rangle$ is spanned by $\dd_{x_1}$, which is transverse to the contact structure along $\kappa$:
$$(\cos(\theta)dx_1-\sin(\theta)dx_2)(\dd_{x_1})=\pm1,\quad \mbox{at }\theta=0,\pi.$$
It evaluates positive for $\theta=0$ and negative for $\theta=\pi$, which corresponds to each of the two branches in the biconormal lift. It is readily verified \cite[Section 3.1]{Geiges08} that $\kappa$ is the transverse push-off, positive {\it and} negative\footnote{The orientation for the negative branch is reversed when considering the global link $\kappa$.}, of $\La(D)$, e.g. observe that the annulus $\{(x_1,x_2,\theta)\in\R^2\times\S^1:x_2=0,0\leq\theta\leq\pi\}$ is a (Weinstein) ribbon for the Legendrian segment $\{(x_1,x_2,\theta)\in\R^2\times\S^1:x_2=0,\theta=\pi/2\}$.
\end{proof}

Proposition \ref{prop:TransPush} implies that real Morsifications $\wt f$ yield models for the Legendrian link $\La_f\sse(\S^3,\xi_\st)$ of a singularity $f\in\C[x,y]$, as introduced in Definition \ref{def:legendrian}. That is, given an isolated plane curve singularity $f\in\C[x,y]$, the Legendrian link $\La_f\sse(\S^3,\xi_\st)$ is Legendrian isotopic to the Legendrian lift $\La(D_{\wt f})\sse(\S^3,\xi_\st)$ of a divide $D_{\wt f}\sse\R^2$ of a real Morsification, and thus we now directly focus on studying the Legendrian links $\La(D_{\wt f})\sse(\S^3,\xi_\st)$.

Let us now prove Theorem \ref{thm:main}. For that, we use N. A'Campo's description \cite{ACampo99} of the set of vanishing cycles associated to a divide of a real Morsification. For each double point $p_i\in D$ in the divide $D=D_{\wt f}$, there is a vanishing cycle $\vartheta_{p_i}$. For each bounded region of $\R^2\setminus D$, which we label by $q_j$, there is a vanishing cycle $\vartheta_{q_j}$. First, we visualize those vanishing cycles by perturbing the divide $D\sse\R^2$ to a divide $D'\sse\R^2$, as depicted in Figure \ref{fig:ACampoVC}.(i) and (ii). The lift of $D'$ only uses one conormal direction at a given point. This perturbation is a front homotopy and thus produces a Legendrian isotopy of the associated Legendrian link.
	
\begin{center}
	\begin{figure}[h!]
		\centering
		\includegraphics[scale=0.75]{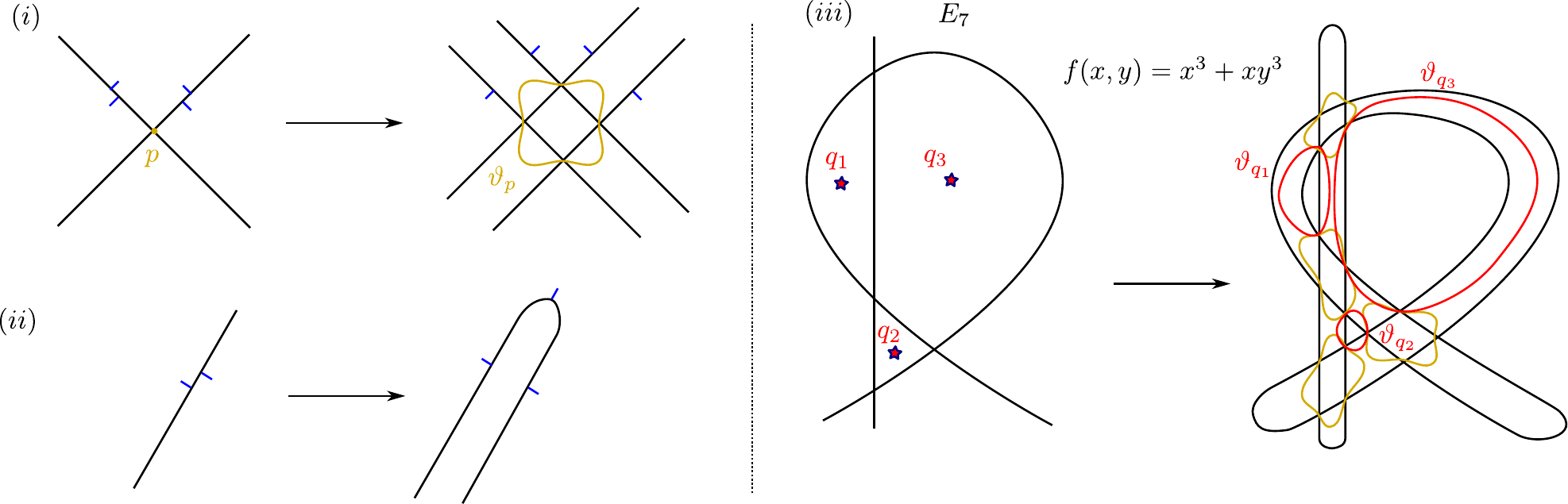}
		\caption{(Left) Two front homotopies from the pieces of a divide to a (generic) Legendrian front. The vanishing cycle $\vartheta_p$ is drawn in the Lagrangian base $\R^2$. (Right) A perturbation of a divide for the $E_7$-singularity. The vanishing cycles $\vartheta_p$ coming from the double points of the divide are drawn in yellow, and the vanishing cycles $\vartheta_q$ coming from each of the three bounded interior regions are drawn in red.}
		\label{fig:ACampoVC}
	\end{figure}
\end{center}

Once the perturbation has been performed, we can draw the curves $\vartheta_{p_i},\vartheta_{q_j}$ as in Figure \ref{fig:ACampoVC}. For instance, Figure \ref{fig:ACampoVC}.(iii) depicts the case of the $E_7$-singularity with a particular choice of divide $D$ and its perturbation $D'$, with $\vartheta_{p_i}$ in yellow and $\vartheta_{q_j}$ in red. That is, for each double point, the curve $\vartheta_{p_i}$ is a closed simple curve exactly through the four new double points in $D'$. For each closed region, $\vartheta_{q_j}$ is a simple closed curve which (exactly) passes through the double points at the perturbed boundary in $D'$ of the region $q_j$. The algorithm in \cite{ACampo99} implies that a {\it singular} model of the topological Milnor fiber of $f$ is obtained as $\R^2$ union the conical Lagrangian conormal $L(D')$ of the perturbed divide $D'$. This Lagrangian conormal intersects the unit cotangent bundle of $T^*\R^2$ at $\La(D')$ and thus, being conical, the information of $L(D')$ is equivalent to that of $\La(D')$. In addition, \cite{ACampo99} guarantees that the curves $\vartheta_{p_i},\vartheta_{q_j}$ are vanishing cycles for the real Morsification $\wt f$.

At this stage, the key fact that we use from A'Campo's algorithm is that our choice of immersion of the divide $D'\sse\R^2$, given by the perturbation, exhibits Lagrangian 2-disks $\D^2_{p_i},\D^2_{q_j}\sse\R^2$ such that $\dd \D^2_{p_i}=\vartheta_{p_i}$ and $\dd \D^2_{q_j}=\vartheta_{q_j}$. For $\vartheta_{p_i}$, this follows from Figure \ref{fig:ACampoVC}.(i), where the 2-disk $\D^2_{p_i}$ is (a small extension of) the square given by the four double points in $D'$ appearing in the perturbation of $p_i\in D$. For $\vartheta_{q_j}$, the 2-disk $\D^2_{q_j}$ is chosen to be a small extension of the bounded region itself. These disks are (exact) Lagrangian because $\R^2\sse(T^*\R^2,\la_\st)$ is exact Lagrangian. The Liouville vector field in $(T^*\R^2,\la_\st)$ vanishes at $\R^2$ and is tangent to $L(D')$. Hence, the inverse flow of the Liouville field retracts the Weinstein pair $(\R^4,\La(D'))$ to $L(D')$ union the zero section $\R^2$. This shows that $L(D')\cup\R^2$ is a Lagrangian skeleton of the Weinstein pair $(\R^4,\La(D))$. Now, the Lagrangian skeleton has an open boundary at the unbounded part of $\R^2$, which can be trimmed \cite{Starkston} to the disks $\D^2_{p_i},\D^2_{q_j}\sse\R^2$. Thus, the union of the conical Lagrangian $L(D')$ and the Lagrangian  2-disks $\D^2_{p_i},\D^2_{q_j}\sse\R^2$ is a Lagrangian skeleton of the Weinstein pair $(\R^4,\La(D'))$, as required.\hfill$\Box$

\subsection{Lagrangian Skeleta} Arboreal Lagrangian skeleta $\L\sse(W,\la)$ for Weinstein 4-manifolds are defined in \cite{Nadler17,Starkston}. Given a Weinstein manifold $W=W(\La)$, the arborealization procedure in \cite{Starkston} yields an arboreal Lagrangian skeleton $\L\sse(W,\la)$ with $\dd\L\neq\varnothing$. Intuitively, those Lagrangian skeleta are obtained by attaching 2-handles to $\D^2$ along a (modification of a) front for $\La$, and thus roughly contain the same information as a front $\pi(\La)\sse\R^2$ for $\La$. Let $\La\sse(\S^3,\xi_\st)$ be a Legendrian link and $(W,\la)$ a Weinstein manifold.

\begin{definition} 
A compact arboreal Lagrangian skeleton $\L\sse\C^2$ for a Weinstein pair $(\C^2,\La)$ is said to be closed if $\dd\L=\La$. A compact arboreal Lagrangian skeleton $\L\sse W$ for a Weinstein manifold $(W,\la)$ is said to be closed if $\dd\L=\varnothing$.
\end{definition}

The Lagrangian skeleta in Theorem \ref{thm:main} and Corollary \ref{cor:main} are arboreal and closed. For reference, we denote the two Cal-skeleta associated to a real Morsification $\tilde{f}$ of an isolated plane curve singularity $f\in\C[x,y]$ by
$$\displaystyle\L(\tilde f):=M_f\cup_{\vartheta(\tilde f)}\displaystyle\bigcup_{i=1}^{|\vartheta(\tilde f)|} \D^2,\qquad \displaystyle\overline{\L}( \tilde f):=\overline{M}_f\cup_{\vartheta(\tilde f)}\displaystyle\bigcup_{i=1}^{|\vartheta(\tilde f)|} \D^2.$$
The former $\L(\tilde f)$ is a Lagrangian skeleton for the Weinstein pair $(\C^2,\La_f)$, and the latter for the Weinstein 4-manifold $W(\La_f)$. The notation $\overline{M}_f$ stands for the surface obtained by capping each of the boundary components of the Milnor fiber $M_f$ with a 2-disk. The notation $\L(f)$ and $\overline{\L}(f)$ will stand for any Cal-skeleton obtained from {\it a} real Morsification $\tilde{f}$ as in Theorem \ref{thm:main} and Corollary \ref{cor:main}.

\begin{remark} In the context of low-dimensional topology, the 2-complexes underlying these Lagrangian skeleta are often referred to as {\it Turaev's shadows}, following \cite[Chapter 8]{Turaev94}. In particular, it is known how to compute the signature of a (Weinstein) 4-manifold from any Cal-skeleton by using \cite[Chapter 9]{Turaev94}. Similarly, the $SU(2)$-Reshetikhin-Turaev-Witten invariant of the 3-dimensional (contact) boundary can be computed with the state-sum formula in \cite[Chapter 10]{Turaev94}. It would be interesting to explore if such combinatorial invariants can be enhanced to detect information on the contact and symplectic structures.\hfill$\Box$
\end{remark}

\section{Augmentation Stack and The Cluster Algebra of Fomin-Pylyavskyy-Shustin-Thurston}\label{sec:cluster} In the article \cite{Morsification_FPST}, the authors develop a connection between the topology of an isolated singularity $f\in\C[x,y]$ and the theory of cluster algebras. In concrete terms, they associate a cluster algebra $A(f)$ to an isolated singularity. An initial cluster seed for $A(f)$ is given by a quiver $Q(D_{\wt f})$ coming from the A$\Gamma$-diagrams of a divide $D_{\wt f}$ of a real Morsification of $f$. Equivalently, by \cite{ACampo99,GuseinZade74}, the quiver $Q(D_{\wt f})$ is the intersection quiver for a set of vanishing cycles associated to a real Morsification of $f$. The conjectural tenet in \cite{Morsification_FPST} is that different choices of Morsifications lead to mutation equivalent quivers and, conversely, two quivers associated to two real Morsifications of the {\it same} complex topological singularity must be mutation equivalent.

There are two varieties associated to a cluster algebra, the $\mathcal{X}$-cluster variety and the $\SA$-cluster variety \cite{FockGoncharov_ModuliLocSys,Goncharov17,ShenWeng19}. In the case of the cluster algebra $A(f)$ from \cite{Morsification_FPST}, one can ask whether either of these varieties has a particularly geometric meaning. Our suggestion is that either of these cluster varieties is the moduli space of {\it exact} Lagrangian fillings for the Legendrian knot $\La_f\sse(\R^3,\xi_\st)$. Equivalently, they are the moduli space of (certain) objects of a Fukaya category associated to the Weinstein pair $(\C^2,\La_f)$; for instance, the partially wrapped Fukaya category of $\C^2$ stopped at $\La_f$. In this sense, these cluster varieties are mirror to the Weinstein pair $(\R^4,\La_f)$.\footnote{The difference between $\mathcal{X}$- and $\mathcal{A}$-varieties should be the {\it decorations} we require for the Lagrangian fillings.} Focusing on the Legendrian link $\La_f\sse(\R^3,\xi_\st)$, let us then   suggest an alternative route from a plane curve singularity $f\in\C[x,y]$ to a cluster algebra $\SA(f)$, following Definition \ref{def:legendrian} and Proposition \ref{prop:unique} and \ref{prop:TransPush}.

Starting with $f\in\C[x,y]$, consider the Legendrian\footnote{In the context of plabic graphs \cite[Section 6]{Morsification_FPST}, the zig-zag curves \cite{Goncharov17,Postnikov06} also provide a front for the Legendrian link $\La_f$.} $\La_f\sse(\R^3,\xi_\st)$, where $(\R^3,\xi_\st)$ is identified as the complement of a point in $(\S^3,\xi_\st)$ and the Legendrian DGA $\A(\La_f)$, as defined by Y. Chekanov in \cite{Chekanov02} and see \cite{EtnyreNg18}. Then we define $\SA(f)$ to be the coordinate ring of functions on the {\it augmentation} variety $\SA(\La_f)$ of the DGA $\A(\La_f)$. Technically, the DGA $\A(\La_f)$ allows for a choice of base points, and the augmentation variety depends on that. Thus, it is more accurate to define:

\begin{definition}\label{def:augmentation} Let $f\in\C[x,y]$ be an isolated singularity, the augmentation algebra $\SA(f)$ associated to $f$ is the ring of $k$-regular functions on the moduli stack of objects $\ob(\Aug_+(\La_f))$ of the augmentation category $\Aug_+(\La_f)$.\hfill$\Box$
\end{definition}

The $\Aug_+(\La)$ augmentation category of a Legendrian link $\La\sse(\R^3,\xi_\st)$ is introduced in \cite{NRSSZ}. An exact Lagrangian filling\footnote{Throughout the text, exact Lagrangian fillings are, if needed, implicitely endowed with a $\C^*$-local system.} defines an object in the category $\Aug_+(\La)$, and the morphisms between two such objects are given by (a linearized version of) Lagrangian Floer homology. In fact, there is a sense in which any object in $\Aug_+(\La)$ comes from a Lagrangian filling \cite{PanRutherford19,PanRutherford20}, possibly immersed, and thus $\ob(\Aug_+(\La))$ is a natural candidate for a moduli space of Lagrangian fillings. The algebra $\SA(f)$ is known to be a cluster algebra \cite{GSW20_1} in characteristic two. The lift to characteristic zero can be obtained by combining \cite{CasalsNg20} and \cite{GSW20_1}.

By Proposition \ref{prop:unique}, $\SA(f)$ is a well-defined invariant of the complex topological singularity. For these Legendrian links $\La=\La_f$, the Couture-Perron algorithm \cite{CouturePerron} implies that there exist a Legendrian front $\pi(\La_f)\sse\R^2$ given by the $(-1)$-closure of a positive braid $\beta\Delta^2$, where $\Delta$ is the full twist; equivalently the front is the rainbow closure of the positive braid $\beta$ \cite{CasalsHonghao}. Hence, there is a set of non-negatively graded Reeb chords generating the DGA $\A(\La_f)$ and $\ob(\Aug_+(\La_f))$ coincides with the set of $k$-valued augmentations of $\A(\La_f)$ where exactly {\it one} base point per component has been chosen, $k$ a field. The articles \cite{CasalsNg20,Kalman} provide an explicit and computational model for $\ob(\Aug_+(\La_f))$, and thus $\SA(f)$, as follows.

First, suppose that $\La=\La_f$ is a knot. Then, $\SA(f)$ is the algebra of regular functions of the affine variety
$$X(\beta):=\{\SB(\beta\Delta^2)+\mbox{diag}_{i(\beta)}(t,1,,\ldots,1)=0\}\sse\C^{|\beta\Delta^2|+1},$$
where $\SB$ are the ($i(\beta)\times i(\beta)$)-matrices defined in \cite[Section 3]{CasalsNg20} and Computation \ref{computation} below, $i(\beta)$ is the number of strands of $\beta,\Delta$, and $|\beta\Delta^2|$ is the number of crossings of $\beta\Delta^2$. In the case $\La_f$ is a {\it link} with $l$ components, the space $\ob(\Aug_+(\La_f))$ is a stack\footnote{Namely, it is isomorphic to a quotient of $X(\beta)\times(\C^*)^l$ by a non-free $(\C^*)^{l-1}$-action. }, with isotropy groups of the form $(\C^*)^k$. If the tenet \cite[Conjecture 5.5]{Morsification_FPST} holds, the affine algebraic type of the augmentation stack $\ob(\Aug_+(\La_f))$ of a Legendrian link should recover the Legendrian link $\La_f$ and the complex topological type of the singularity $f$. Here is how to compute $\ob(\Aug_+(\La_f))$.\\

\begin{computation}\label{computation} Let $\La=\La_f$ be an algebraic knot, we can find a set of equations for the affine variety $\ob(Aug_+(\La_f))$, essentially using \cite{Kalman06_PositiveBraid}, see also \cite{CasalsNg20}. Consider a positive braid\footnote{Note that $\beta^\circ$ can be written in the form $\beta^\circ=\beta\Delta^2$.} $\beta^\circ\in\mbox{Br}^+_n$ such that the $(-1)$-closure of $\beta^\circ$ is a front for $\La=\La(\beta^\circ)$. For $k\in[1,n-1]$, define the following $n\times n$ matrix $P_k(z)$, with variable $z\in\C$:
\[
(P_k(z))_{ij} = \begin{cases} 1 & i=j \text{ and } i\neq k,k+1 \\
1 & (i,j) = (k,k+1) \text{ or } (k+1,k) \\
z & i=j=k+1 \\
0 & \text{otherwise;}
\end{cases}
\]
Namely, $P_k(z)$ is the identity matrix except for the $(2\times 2)$-submatrix given by rows and columns $k$ and $k+1$, where it is $\left( \begin{smallmatrix} 0 & 1 \\ 1 & z \end{smallmatrix} \right)$. Suppose that the crossings of $\beta^\circ$, left to right, are $\sigma_{k_1},\ldots,\sigma_{k_s}$, $s=|\beta^\circ|\in\N$, $\sigma_i\in\mbox{Br}^+_{n}$ the Artin generators. Then the augmentation stack $\ob(\Aug_+(\La_f))$ is cut out in $\C^{s}\times\C^*=\Spec[z_1,z_2,\ldots,z_s,t,t^{-1}]$ by the $n^2$ equations
\begin{equation}\label{eq:aug}
\mbox{diag}_n(t,1,1,\ldots,1) + P_{k_1}(z_1)P_{k_2}(z_2)\cdots P_{k_s}(z_s)=0.
\end{equation}
The matrix $P_{k_1}(z_1)P_{k_2}(z_2)\cdots P_{k_s}(z_s)$ is denoted by $\SB(\beta^\circ)$. Equations \ref{eq:aug} provide a computational mean to an explicit description of the affine varieties $\ob(\Aug_+(\La_f))$ that yield the cluster algebra $\SA(f)$.\hfill$\Box$
\end{computation}

\begin{example} Consider the plane curve singularity\footnote{We have chosen this example as a continuation of \cite[Example 5.3]{CouturePerron} and \cite[Figure 6]{Morsification_FPST}.} described by
$$f(x,y)=-12 x^{10} y^2-4 x^9 y^2-2 x^7 y^4+6 x^6 y^4-4 x^3 y^6+x^{14}-2 x^{13}+x^{12}+y^8=$$
$$=\left(2 x^3 y^2-4 x^5 y+x^7-x^6-y^4\right) \left(2 x^3 y^2+4 x^5 y+x^7-x^6-y^4\right)$$
The Puiseux expansion yields $y(x)=x^{3/2}+x^{7/4}$ and using the Couture-Perron algorithm \cite{CouturePerron}, or \cite[Definition 11.3]{Morsification_FPST}, a positive braid word associated to this singularity is
$$\beta=(\sigma_2\sigma_1\sigma_3\sigma_2\sigma_1\sigma_3\sigma_2\sigma_1)\sigma_3(\sigma_1\sigma_2\sigma_3\sigma_1\sigma_2\sigma_3\sigma_1\sigma_2)\sigma_1\sigma_3$$
The Legendrian $\La_f\sse(\R^3,\xi)$ is the rainbow closure of $\beta$, and the $(-1)$-framed closure of $\beta^\circ=\beta\Delta^2$. Note that $\La_f$ is a knot, and thus we will use one base point $t\in\C^*$ in the computation of $X(\beta)=\ob(\Aug_+(\La_f))$. Following Computation \ref{computation} above, we can write equations for affine variety $X(\beta)$ as a subset $X(\beta)\sse\C^{31}\times\C^*$. We use coordinates $(z_1,z_2,\ldots,z_{31};t)\in\C^{31}\times\C^*$, $(z_1,z_2,\ldots,z_{19})$ corresponding to the $19$ crossings of $\beta$ and $(z_{20},\ldots,z_{31})$ account for the $12$ crossings of $\Delta^2\in\Br^+_3$. There are a total of 16 equations, the first three of which read as follows:

$$z_{11}+z_9 z_{12}+\left(z_9+\left(z_{11}+z_9 z_{12}\right) z_{18}\right) z_{20}+\left(z_{13}+z_9 z_{14}+\left(z_{11}+z_9 z_{12}\right) z_{15}\right) z_{21}+$$
$$\left(z_9 z_{16}+\left(z_{11}+z_9 z_{12}\right) z_{17}+\left(z_{13}+z_9 z_{14}+\left(z_{11}+z_9 z_{12}\right) z_{15}\right) z_{19}+1\right) z_{23}=-t^{-1}$$

$$z_7+z_6 z_9+(z_8 z_{10}+z_6 z_{11}+(z_7+z_6 z_9) z_{12}+1) z_{18}+(z_8+z_6 z_{13}+(z_7+z_6 z_9) z_{14}+$$
$$(z_8 z_{10}+z_6 z_{11}+(z_7+z_6 z_9) z_{12}+1) z_{15}) z_{22}+(z_6+(z_7+z_6 z_9) z_{16}+(z_8 z_{10}+z_6 z_{11}+(z_7+z_6 z_9) z_{12}+1) z_{17}+$$
$$(z_8+z_6 z_{13}+(z_7+z_6 z_9) z_{14}+(z_8 z_{10}+z_6 z_{11}+(z_7+z_6 z_9) z_{12}+1) z_{15}) z_{19}) z_{24}+(z_8 z_{10}+z_6 z_{11}+(z_7+z_6 z_9) z_{12}+$$
$$(z_7+z_6 z_9+(z_8 z_{10}+z_6 z_{11}+(z_7+z_6 z_9) z_{12}+1) z_{18}) z_{20}+(z_8+z_6 z_{13}+(z_7+z_6 z_9) z_{14}+$$
$$(z_8 z_{10}+z_6 z_{11}+(z_7+z_6 z_9) z_{12}+1) z_{15}) z_{21}+(z_6+(z_7+z_6 z_9) z_{16}+(z_8 z_{10}+z_6 z_{11}+(z_7+z_6 z_9) z_{12}+1) z_{17}+$$
$$(z_8+z_6 z_{13}+(z_7+z_6 z_9) z_{14}+(z_8 z_{10}+z_6 z_{11}+(z_7+z_6 z_9) z_{12}+1) z_{15}) z_{19}) z_{23}+1) z_{31}=0$$

$$z_1 z_8+(z_2+z_1 z_6) z_{13}+(z_2 z_9+z_1 (z_7+z_6 z_9)+1) z_{14}+(z_8 z_{10} z_1+z_1+(z_2+z_1 z_6) z_{11}+$$
$$(z_2 z_9+z_1 (z_7+z_6 z_9)+1) z_{12}) z_{15}+(z_2+z_1 z_6+(z_2 z_9+z_1 (z_7+z_6 z_9)+1) z_{16}+(z_8 z_{10} z_1+z_1+(z_2+z_1 z_6) z_{11}+$$
$$(z_2 z_9+z_1 (z_7+z_6 z_9)+1) z_{12}) z_{17}+(z_1 z_8+(z_2+z_1 z_6) z_{13}+(z_2 z_9+z_1 (z_7+z_6 z_9)+1) z_{14}+(z_8 z_{10} z_1+z_1+$$
$$(z_2+z_1 z_6) z_{11}+(z_2 z_9+z_1 (z_7+z_6 z_9)+1) z_{12}) z_{15}) z_{19}) z_{25}+(z_1 z_7+(z_2+z_1 z_6) z_9+(z_8 z_{10} z_1+z_1+$$
$$(z_2+z_1 z_6) z_{11}+(z_2 z_9+z_1 (z_7+z_6 z_9)+1) z_{12}) z_{18}+(z_1 z_8+(z_2+z_1 z_6) z_{13}+(z_2 z_9+z_1 (z_7+z_6 z_9)+1) z_{14}+$$
$$(z_8 z_{10} z_1+z_1+(z_2+z_1 z_6) z_{11}+(z_2 z_9+z_1 (z_7+z_6 z_9)+1) z_{12}) z_{15}) z_{22}+(z_2+z_1 z_6+(z_2 z_9+z_1 (z_7+z_6 z_9)+1) z_{16}+$$
$$(z_8 z_{10} z_1+z_1+(z_2+z_1 z_6) z_{11}+(z_2 z_9+z_1 (z_7+z_6 z_9)+1) z_{12}) z_{17}+(z_1 z_8+(z_2+z_1 z_6) z_{13}+(z_2 z_9+$$
$$z_1 (z_7+z_6 z_9)+1) z_{14}+(z_8 z_{10} z_1+z_1+(z_2+z_1 z_6) z_{11}+(z_2 z_9+z_1 (z_7+z_6 z_9)+1) z_{12}) z_{15}) z_{19}) z_{24}+1) z_{28}+$$
$$(z_8 z_{10} z_1+z_1+(z_2+z_1 z_6) z_{11}+(z_2 z_9+z_1 (z_7+z_6 z_9)+1) z_{12}+(z_1 z_7+(z_2+z_1 z_6) z_9+(z_8 z_{10} z_1+z_1+$$
$$(z_2+z_1 z_6) z_{11}+(z_2 z_9+z_1 (z_7+z_6 z_9)+1) z_{12}) z_{18}+1) z_{20}+(z_1 z_8+(z_2+z_1 z_6) z_{13}+(z_2 z_9+z_1 (z_7+z_6 z_9)+1) z_{14}+$$
$$(z_8 z_{10} z_1+z_1+(z_2+z_1 z_6) z_{11}+(z_2 z_9+z_1 (z_7+z_6 z_9)+1) z_{12}) z_{15}) z_{21}+(z_2+z_1 z_6+(z_2 z_9+z_1 (z_7+z_6 z_9)+1) z_{16}+$$
$$(z_8 z_{10} z_1+z_1+(z_2+z_1 z_6) z_{11}+(z_2 z_9+z_1 (z_7+z_6 z_9)+1) z_{12}) z_{17}+(z_1 z_8+(z_2+z_1 z_6) z_{13}+(z_2 z_9+$$
$$z_1 (z_7+z_6 z_9)+1) z_{14}+(z_8 z_{10} z_1+z_1+(z_2+z_1 z_6) z_{11}+(z_2 z_9+z_1 (z_7+z_6 z_9)+1) z_{12}) z_{15}) z_{19}) z_{23}) z_{30}=0$$

The remaining 13 equations are longer, but can be readily obtained. This hopefully illustrates that the method is computationally immediate.\footnote{Even if the equations themselves, being rather long, may not be particularly enlightening.}\hfill$\Box$
\end{example}

\begin{remark}
	\begin{itemize}
		\item[(i)] One may consider the moduli stack $\ob(\Sh^s_{\La_f}(\R^2))$ of sheaves with microlocal rank-1 along $\La_f$, instead of $\ob(\Aug_+(\La_f))$. By \cite{NRSSZ}, there is an equivalence of categories $\Aug_+(\La_f)\cong\Sh^1_{\La_f}(\R^2)$. The stack $\ob(\Sh^1_{\La_f}(\R^2))$ is a $\mathcal{X}$-cluster variety; the associated $\SA$-cluster variety in the cluster ensemble is the moduli of {\it framed} sheaves \cite{ShenWeng19}.\footnote{The cluster algebra structure for $\SA(f)$ defined by \cite{GSW20_1} is obtained by pulling-back the cluster algebra structure of the open Bott-Samelson cell associated to $\beta$. There should exist a cluster algebra structure on $\SA(f)$ defined strictly in Floer-theoretical terms.} In short, the cluster algebra $\SA(f)$ could have been defined in terms of the moduli space of constructible sheaves microlocally supported in $\La$, instead of Floer theory.\\
		
		\item[(ii)] The $\Aug_+$-category is Floer-theoretical in nature, e.g. its morphisms are certain Floer homology groups. It would have also been natural to consider the partially wrapped Fukaya category $W(\C^2,\La_f)$, as defined \cite{GPS1,Sylvan19_PartiallyWrapped}, or the infinitesimal Fukaya category $Fuk(\C^2,\La)$ \cite{NadlerZaslow09,Nadler09}. These are Floer-theoretical Legendrian invariants associated to $\La_f$, and thus the singularity $f\in\C[x,y]$, which might be of interest on their own.
	\end{itemize} 
\end{remark}

\section{A few Computations and Remarks}\label{sec:computations}

Consider the derived dg-category $\Sh_\La(M)$ of constructible sheaves in a closed smooth manifold $M$ microlocally supported at a Legendrian link $\La\sse(\dd T^\infty M,\xi_\st)$, e.g. as introduced in \cite[Section 1]{STZ}. Equivalently, one may consider a conical Lagrangian $L\sse T^*M$ instead of $\La\sse(T^\infty M,\xi_\st)$; in practice, the input data is a wavefront $\pi(\La)\sse M$ \cite{Arnold90}. Let $\mush$ denote the sheaf of microlocal sheaves defined\footnote{Thanks go to V. Shende for helpful discussions on sheaf invariants.} in \cite[Section 5]{NadlerShende20}. There are two situations we consider, depending on whether the focus is on the Weinstein pair $(\C^2,\La_f)$ or on the Weinstein 4-manifold $W(\La_f)$:

\begin{itemize}
	\item[(i)] {\bf Sheaf Invariants of the Weinstein pair $(\C^2,\La_f)$.}\footnote{Invariance up to Weinstein homotopy \cite{CieliebakEliashberg12}, and also symplectomorphism of Liouville pairs.} The category of microlocal sheaves $\mush(\L(f))$ is an invariant of $(\C^2,\La_f)$, as established in \cite{GKS_Quantization,NadlerShende20,STZ}.\footnote{The category $\mush(\L(f))$ is likely {\it not} an invariant of the Weinstein 4-manifold $W(\La_f)$ itself.} In this case, the global sections $\mush(\L(f))$ is a category equivalent to the more familiar $\Sh_{\La(f)}(\R^2)$. For simplicity, we focus on the moduli stack $\SS(f)\sse \Ob(\Sh_{\La(f)}(\R^2))$ of {\it simple} sheaves, whose microlocal support is rank one, microlocally supported in the Legendrian link of an isolated plane curve singularity $f:\C^2\lr\C$. See \cite[Section 7.5]{KashiwaraSchapira_Book} or \cite[Section 1.10]{GKS_Quantization} for a detailed discussion on simple sheaves. In our case $\La=\La(f)$, $\SS(f)$ is an Artin stack of finite type \cite[Prop. 5.20]{STZ}, and typically is an algebraic variety or a $G$-quotient thereof, with $G=(\C^*)^k$ or $\GL(k,\C)$. Note that $\mush(\L(f))$ is equivalent to the wrapped Fukaya category of $\C^2$ stopped at $\La_f$ \cite{GPS3}.\\
	
	\item[(ii)] {\bf Sheaf Invariants of the Weinstein 4-manifold $W(\La_f)$.} The category $\mush(\overline{\L}(f))$ of microlocal sheaves \cite{NadlerShende20} on a Lagrangian skeleton $\overline{\L}(f)\sse W(\La_f)$ is an invariant of $W(\La_f)$, up to Weinstein homotopy \cite{NadlerShende20} and up to symplectomorphism \cite{GPS3}. This category is $\Sh_{\vartheta(f)}(\overline{M}_f)$, or $\mu loc(\overline{\L}(f))$, in the notation of \cite{STW_Combinatorics}, i.e. the global sections of the Kashiwara-Schapira sheaf of dg-categories \cite[Prop. 3.5]{STW_Combinatorics} on the Lagrangian skeleton $\overline{\L}(f)$. For simplicity, we focus on the moduli stack $\Theta(f)\sse \mush(\overline{\L}(f))$ of simple sheaves as well. Note that $\mush(\overline{\L}(f))$ is equivalent to the wrapped Fukaya category of $W(\La_f)$ by \cite{GPS3}.
\end{itemize} 

The moduli stack $\SS(f)$ in (i) is isomorphic to the stack of simple sheaves in $\Ob(\Sh_{\vartheta(f)}(M_f))$. This is because the union of $\R^2\sse T^*\R^2$ and the Lagrangian cone of $\La\sse(T^+\R^2,\xi_\st)$ is a Lagrangian skeleton for the relative Weinstein pair $(\C^2,\La)$, so is $\L(f)$ by Theorem \ref{thm:main}, and $\Ob(\Sh_{\vartheta(f)}(M_f))$ is an invariant of the Weinstein pair $(\C^2,\La)$, independent of the choice of Lagrangian skeleton. Thus, the difference between $\SS(f)$ and $\Theta(f)$ is at the boundary, which for $\SS(f)$ might give monodromy contributions (and these become trivial on $\Theta(f)$). In other words, since $\overline{\L}(f)$ is obtained from $\L(f)$ by attaching 2-disks (to close the boundary of the Milnor fiber $M_f$), the category $\mush(\overline{\L}(f))$ is a homotopy pull-back of $\mush(\L(f))$. In particular, the moduli stacks of simple microsheaves are related as above. 

\begin{remark} There are currently two methods for computing $\SS(f)$: either by direct means, as exemplified in \cite{STZ}, or by using the equivalence of categories $Aug_+(\La(f))\cong\Sh^{s}_{\La_f}(\R^2)$ from \cite[Theorem 1.3]{NRSSZ}, the latter being denoted by $\mathcal{C}_1(\La_f)$ in \cite{NRSSZ}. Thanks to the computational techniques available for augmentation varieties, the moduli of objects $\Ob(Aug_+(\La(f)))$ is readily computable for $(-1)$-framed closures of positive braids as in Section \ref{sec:cluster} above, confer Computation \ref{computation}. Similarly $\Theta(f)$ could be computed directly, or by means of the isomorphism to the wrapped Fukaya category\footnote{Should the reader be willing to use the surgery formula, this wrapped Fukaya category may be presented as modules over the Legendrian DGA of $\La_f$. (This is only informative and not needed for the present purposes.)} of $W(\La_f)$.\hfill$\Box$
\end{remark}

In this section, we take to opportunity to build on \cite{NadlerShende20,STW_Combinatorics} and perform an actual computation for a class of Cal-Skeleta coming from Theorem \ref{thm:main}.

\begin{center}
	\begin{figure}[h!]
		\centering
		\includegraphics[scale=0.7]{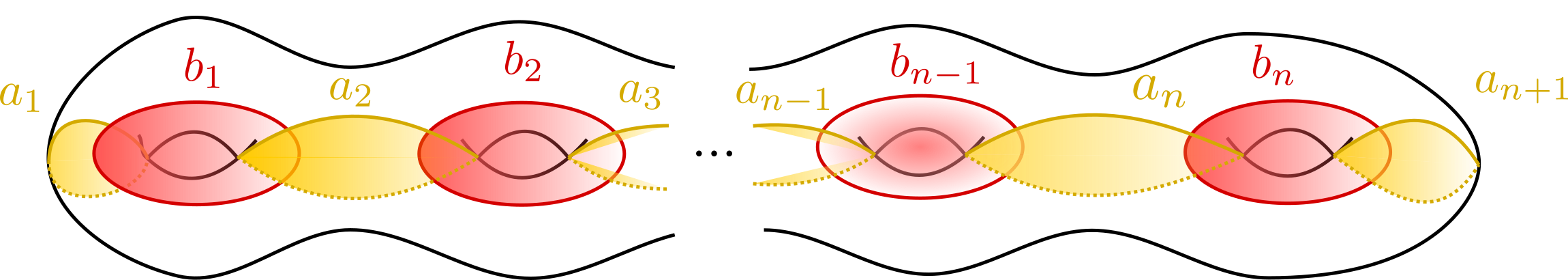}
		\caption{A Cal-skeleta $\overline{\L}(f_{2n+1})$ for the Weinstein 4-manifolds $W(\La(A_{2n+1}))$.}
		\label{fig:AnSkeleta}
	\end{figure}
\end{center}

\subsection{Cal-Skeleta for $A_n$-Singularities} Consider the $A_n$-singularity $f_n(x,y)=x^{n+1}+y^2$. The Legendrian $\La(A_n)\sse(\R^3,\xi_\st)$ associated to the singularity is the max-tb Legendrian $(2,n+1)$-torus link. By Theorem \ref{thm:main}, a Lagrangian skeleton $\L(f_n)$ for the Weinstein pair $(\C^2,\La_f)$ is obtained by attaching $n$ 2-disks to a $(3/2-(-1)^{n}/2)$--punctured $\floor{\frac{n-1}{2}}$--genus surface along an $A_n$-Dynkin chain of embedded curves. Similarly, Corollary \ref{cor:main} implies that a Lagrangian skeleton $\overline{\L}(f_n)$ for the Weinstein 4-manifold $W_n=W(\La(A_n))$ is given by attaching $n$ 2-disks to a $\floor{\frac{n-1}{2}}$--genus surface along an $A_n$-Dynkin chain, as depicted in orange in Figure \ref{fig:AnSheaves}, see also Figure \ref{fig:AnSkeleta}.

\begin{center}
	\begin{figure}[h!]
		\centering
		\includegraphics[scale=0.7]{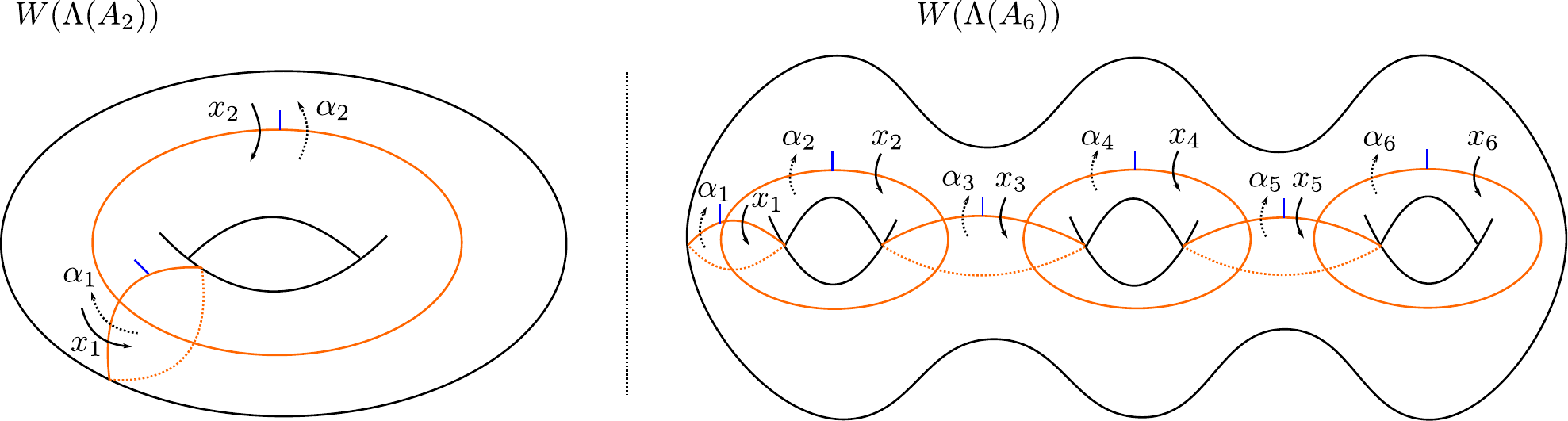}
		\caption{The Cal-skeleta $\overline{\L}(f)$ for the Weinstein 4-manifolds $W(\La(A_2))$ and $W(\La(A_6))$. The relative Cal-skeleta $\overline{\L}(f)$ for the corresponding Weinstein pairs $(\C^2,\La(A_2))$ and $(\C^2,\La(A_6))$ are obtained by introducing one puncture to the surfaces.}
		\label{fig:AnSheaves}
	\end{figure}
\end{center}

Let us compute $\Theta(f_n)$ for $n\in\N$ even, so that $\La(A_n)$ is a knot; the $n\in\N$ odd case is similar. The key technical tool is the Disk Lemma \cite[Lemma 4.2.3]{Karabas18}. The complement $\overline{M}_f\setminus \vartheta(f)$ of the vanishing cycles is a 2-disk, and the category of local systems is just $\C$-mod. Thus, the moduli of simple constructible sheaves on $\overline{M}_f$ microlocally supported on (the Legendrian lift of) the vanishing cycles $\vartheta(f)$ consists of a vector space $V=\C$ and maps $x_1,x_2,\ldots,x_n\in End(V)$, one associated to each vanishing cycle. This is depicted in Figure \ref{fig:AnSheaves} for $n=2,6$, and note that $n=|\vartheta(f)|$. Denote by $\overline{\L}(f_n)_0\sse T^*\overline{M}_f$ the Lagrangian skeleton given by $\overline{M}_f$ union the conormal lifts of $\vartheta(f)$. These maps are {\it not} necessarily invertible in $\mush(\overline{\L}(f_n)_0)$.

The skeleton $\overline{\L}(f_n)$ is obtained by attaching $n$ Lagrangian 2-disks to $\overline{\L}(f_n)_0$, i.e. $\overline{\L}(f_n)$ is the homotopy push-out of $\overline{\L}(f_n)_0$ and the disjoint union of $n$ 2-disks. In consequence, the category of microlocal sheaves on $\overline{\L}(f_n)$ is given by the homotopy pull-back of the category of microlocal sheaves on $\overline{\L}(f_n)_0$ and the category of microlocal sheaves on $n$ disjoint 2-disks (which are just copies of $\C$-mod). Attaching a 2-disk along a vanishing $V_i$ cycle in $\vartheta(f)$, $i\in[1,n]$, has the effect of trivializing the ``monodromy'' corresponding map $x_i$, as explained in \cite{STW_Combinatorics} and \cite[Section 4.2]{Karabas18}. Here, the monodromy\footnote{We had written ``monodromy'' in quotations because it is not a priori necessarily invertible.} is given by restricting a microlocal sheaf to (an arbitrarily small neighborhood of) $V_i$. Note that in this restriction, we land into a 1-dimensional Lagrangian skeleton given by a circle $V_i\cong S^1$ {\it union} conical segments coming from the adjacent vanishing cycles. Let us call $\gamma_i$ the composition of maps from $cone(x_i)$ to itself obtained by going around $V_i$, each of the maps coming from traversing a segment. Then, the trivialization is a homotopy to the identity, and it translates into adding a map $\alpha_i$ such that $x_i\alpha_i-1=\gamma_i$.

\begin{example} Consider the map $x_1$ in Figure \ref{fig:AnSheaves} (Left), which is depicted transversely to the vanishing cycle $V_1$. The restriction of a microlocal sheaf to a neighborhood of $V_1$ gives a microlocal sheaf for the skeleton $\S^1\cup T_p^{*,+}\S^1\sse T^*\S^1$, where $T_p^{*,+}\S^1$ is the positive half of the cotangent fiber at a point $p\in\S^1$. Such a microlocal sheaf is described by a (complex of) vector space(s) and an endomorphism. In this case the vector space is $V=\C$ and this endomorphism is identified with $\gamma_1=x_2$. Hence, trivializing along $V_1$ adds a map $\alpha_1\in End(\C)$, which we can think of as a variable $\alpha_1\in\C$, such that $x_1\alpha_1+1=-x_2$. Similarly, trivializing along $V_2$, with $\gamma_2=-\alpha_1$, adds a variable $\a_2\in\C$ such that $1+x_2\a_2=-\a_1$. Hence $\Theta(f)$ is the affine variety
$$\Theta(f_3)=\{(x,y,z)\in\C^3:xyz+x-z-1=0\}.$$
This affine variety appears in the study of isomonodromic deformations of the Painlev\'e I equation \cite[Section 3.10]{vanderPutSaito09}, see also \cite[Section 5]{Boalch18}. \hfill$\Box$
\end{example}

The vanishing cycles $V_1,V_n$ have simpler monodromies $\gamma_1,\gamma_n$, as they only intersect {\it one} other vanishing cycle. Adding the 2-disks to the skeleton $\overline{\L}(f_n)_0$ along $V_1,V_n$ yields a category of microlocal sheaves whose moduli space of simple objects is described by that of $\overline{\L}(f_n)_0$ and the two equations $x_1\alpha_1+1=-x_2$ and $x_n\alpha_n+1=-\alpha_{n-1}$. For each of the middle vanishing cycles $V_i$, $2\leq i\leq n-1$, we have the monodromy $\gamma_i=\alpha_{i-1}x_{i+1}$. In consequence, attaching the $n$ 2-disks $\overline{\L}(f_n)_0$ along all the curves $V_i$, $i\in[1,n]$, leads to the moduli space
$$\Theta(f)\cong\{(x_i,\a_i)\in(\C^2)^n: x_1\alpha_1+1=-x_2, x_n\alpha_n+1=-\alpha_{n-1}, 1+x_j\a_j=\alpha_{j-1}x_{j+1}, j\in[2,n-1]\}.$$


\begin{remark}
Consider $(n+3)$-tuples of vectors $(v_1,\ldots,v_{n+3})\in\C^2$, modulo $\GL_2(\C)$, the equations for $\Theta(f)$ above can be read directly by writing the $(n+3)$-tuple as

$$\left(\begin{array}{c}
1\\
0\end{array}\right),
\left(\begin{array}{c}
1\\
0\end{array}\right),
\left(\begin{array}{c}
-1\\
x_1\end{array}\right),
\left(\begin{array}{c}
\a_1\\
x_2\end{array}\right),
\left(\begin{array}{c}
\a_2\\
x_3\end{array}\right),
\left(\begin{array}{c}
\a_3\\
x_4\end{array}\right),
\left(\begin{array}{c}
\a_4\\
x_5\end{array}\right),
\ldots,
\left(\begin{array}{c}
\a_{n-1}\\
x_{n}\end{array}\right),
\left(\begin{array}{c}
\a_n\\
-1\end{array}\right),$$
and imposing $v_i\wedge v_{i+1}=1$, where we have use the $GL_2(\C)$ gauge group to trivialize the first two vectors, and one component of the third and last vectors. P. Boalch \cite{Boalch18} names this moduli stack after Y. Sibuya \cite{Sibuya75}. Note that \cite[Section 5]{Boalch18} points out that some of these equations were initially discovered by L. Euler in 1764 \cite{Euler1764}. In the context of open Bott-Samelson cells \cite{ShenWeng19,STWZ}, these spaces appear as the open positroid varieties $\{p\in\mbox{Gr}(2,n+3): P_{i,i+1}(p)\neq0\}$, where $P_{i,j}$ is the Pl\"ucker coordinate given by the minor at the $i$ and $j$ columns, and the index $i$ is understood $\Z/(n+3)$-cyclically.\hfill$\Box$
\end{remark}

Finally, we notice that the cohomology $H^*(\Theta(f),\C)$, or that of $H^*(\SS(f),\C)$, can be an interesting invariant \cite[Section 6]{STZ}. For the case of $A_n$-singularities, we can use the fact that these are actually cluster varieties of $A_n$-type in order to compute their cohomology using \cite{LamSpeyer16}. For $n=2m\in\N$ even, and removing any $\C^*$-factors coming from frozen variables, one obtains that the Abelian graded cohomology group is isomorphic to $\Q[t]/t^{m+1}$, $|t|=2$. In general, the mixed Hodge structure for these moduli spaces can be non-trivial, but for singularities of $A_n$-type, these cohomologies are of Hodge-Tate type, and entirely concentrated in $H^{k,(k,k)}$.

\begin{remark} It would be valuable to understand the relation between sheaf invariants of a singularity $f\in\C[x,y]$, such as $\mush(\L(f))$ and $\mush(\overline{\L}(f))$, and classical invariants from singularity theory \cite{ACampo98,AGZV1,AGZV2}. In particular, it could be valuable to develop more systematic methods to compute $\mush(\L(f))$ and $\mush(\overline{\L}(f))$ both directly and from a divide.\hfill$\Box$
\end{remark}




\section{Structural Conjectures on Lagrangian Fillings}\label{sec:conj}

Let $\La\sse(\S^3,\xi_\st)$ be a max-tb Legendrian link. The classification of embedded exact Lagrangian fillings $L\sse(\D^4,\la_\st)$ with fixed boundary $\La$, up to Hamiltonian isotopy, is a central question. The only Legendrian $\La$ for which a complete classification exists is the standard unknot \cite{EliashbergPolterovich96}. In this case, the standard Lagrangian flat disk is the unique filling: there is precisely {\it one} exact Lagrangian filling, up to Hamiltonian isotopy. The recent developments \cite{CasalsHonghao,CasalsNg20,CasalsZaslow,GSW20_1} show that such finiteness is actually rare: e.g. the max-tb torus links $(n,m)$ admit {\it infinitely} many exact Lagrangian filling, up to Hamiltonian isotopy, if $n,m\geq4$. This final section states and discusses Conjectures \ref{conj:ADE} and \ref{conj:BCFG}, which might help in the classification of exact Lagrangian fillings of Legendrian links.

{\bf Geometric Strategy}. Given $\La\sse(\S^3,\xi_\st)$, we would like to know whether it admits finitely many Lagrangian fillings or not, and in the finite case provide the exact count. Theorem \ref{thm:main} provides insight for the class of Legendrian links $\La\sse(\S^3,\xi_\st)$ that are algebraic links and, more generally, arise from a divide. Indeed, Lagrangian fillings for $\La$ can be constructed by using the Lagrangian skeleta for the Weinstein pair $(\C^2,\La)$ built in the statement. For instance, the inclusion of the Milnor fiber $M_{\wt f}\sse\L_{\wt f}$ provides an exact Lagrangian filling, and performing Lagrangian disk surgeries along the Lagrangian 2-disks in $\L_{\wt f}\setminus M_{\wt f}$, which bound vanishing cycles, will potentially yield new Lagrangian fillings. This strategy can be implemented in certain cases but, in general, one must be able to find an {\it embedded} Lagrangian disk in the new Lagrangian skeleton (with an embedded boundary curve), in order to perform the next Lagrangian disk surgery. Curves being immersed rather than embedded\footnote{Equivalently, the existence of curves with zero algebraic intersection but non-empty geometric intersection.}, might a priori represent a challenge.\footnote{The vanishing cycles can be organized as a quiver $Q$, the additional data of a superpotential $(Q,W)$ should be helpful in solving the disparity between {\it immersed} and {\it embedded} curves in the Milnor fiber.} This geometric scheme has the following algebraic incarnation.

{\bf Algebraic Strategy}.  Consider the intersection quiver $Q_{\vartheta(\wt f)}$ of vanishing cycles for a real Morsification $\wt f$, Lagrangian disk surgeries induce mutations of the quiver \cite{STW_Combinatorics} and the (microlocal) monodromies of a local system serve as cluster $\mathcal{X}$-variables \cite{CasalsZaslow,STWZ}. Thus, the cluster algebra $\SA(Q(f))$ associated to the quiver, as it appears in \cite{Morsification_FPST}, governs {\it possible} exact Lagrangian fillings for the Legendrian link $\La$. That is, a Lagrangian filling $L\sse(\D^4,\la_\st)$ yields a cluster chart for this algebra \cite{GSW20_1,STWZ}, and the Lagrangian skeleta from Theorem \ref{thm:main} provide a geometric realization for the quiver in the form of an exact Lagrangian filling with ambient Lagrangian disks ending on it.

The recent developments \cite{CasalsHonghao,GSW20_1,STW_Combinatorics,STWZ} and the existence of the Lagrangian skeleta in Theorem \ref{thm:main} shyly hint towards the fact that, possibly, Lagrangian fillings {\it are} classified by the cluster algebra $\SA(Q(f))$. That is, every cluster chart in $\SA(Q(f))$ is induced by {\it precisely} one exact Lagrangian filling.\footnote{That is, two Lagrangian fillings inducing the same cluster chart in $\SA(Q(f))$ are Hamiltonian isotopic {\it and} every cluster chart is induced by at least one Lagrangian filling.} It should be emphasized that this is {\it not} known for any $\La\sse(\R^3,\xi_\st)$ except the standard Legendrian unknot. It is possible that the case of the Hopf link $\La(A_1)$ can be solved by building on the techniques in \cite{Rizell19_Whitney}, which classifies exact Lagrangian tori near the Whitney sphere\footnote{See also \cite{CoteRizell20}, which appeared during the writing of this manuscript.}. Having informed the reader on the currently available evidence, the following conjectural guide might be helpful.

\begin{conj}[ADE Classification of Lagrangian Fillings]\label{conj:ADE} Let $\La\sse(\R^3,\xi_\st)$ be the Legendrian rainbow closure of a positive braid such that the mutable part of its brick quiver is connected. Then one of the following possibilities occur:
	
	\begin{itemize}
		\item[1.] $\La$ is smoothly isotopic to the link of the $A_{n}$-singularity.\\ Then $\La$ has precisely $\frac{1}{n+2}{2n+2\choose n+1}$ exact Lagrangian fillings.\\
		
		\item[2.] $\La$ is smoothly isotopic to the link of the $D_{n}$-singularity.\\
		Then $\La$ has precisely $\frac{3n-2}{n}{2n-2\choose n-1}$ exact Lagrangian fillings.\\
		
		\item[3.] $\La$ is smoothly isotopic to the link of the $E_6$, $E_7$ or the $E_8$-singularities.\\
		Then $\La$ has precisely 833, 4160, and 25080 exact Lagrangian fillings, respectively.\\
		
		\item[4.] $\La$ has infinitely many exact Lagrangian fillings.
	\end{itemize}
\end{conj}

The following comments are in order:

\begin{itemize}
	\item[(i)] In \cite{FominZelevinsky03_Finite}, S. Fomin and A. Zelevinsky classify cluster algebras of {\it finite} type. This is an ADE-classification, parallel to the classification of simple singularities \cite{AGZV1}, the Cartan-Killing classification of semisimple Lie algebras, finite crystallographic root systems (via Dynkin diagrams) and the like. Thus, Conjecture \ref{conj:ADE} first states that $\La$ will have {\it finitely} many exact Lagrangian fillings, up to Hamiltonian isotopy, if and only if the associated quiver is ADE.\\
	
	\item[(ii)] The case of $\La=\La_f$ an algebraic link associated to a non-simple singularity $f\in\C[x,y]$ of a plane curve follows from \cite{CasalsHonghao}, and the case of a Legendrian $\La$ with a non-ADE underlying quiver has recently been proven in \cite{GSW20_2}. These approaches are based on the following fact: if there exists an embedded exact Lagrangian cobordism from $\La_-$ to $\La_+$ and $\La_-$ admits infinitely many Lagrangian fillings, then so does $\La_+$. See \cite{CasalsNg20,PanCatalan_AugmentationMap} and \cite[Section 6]{CasalsHonghao}. This itself initiates the quest for finding the {\it smallest} Legendrian link which admits infinitely many exact Lagrangian fillings. At present, if we measure the size of a link $\La$ as $\pi_0(\La)+g(\La)$, $g(\La)$ the (minimal) genus of a (any) embedded Lagrangian filling, the smallest known Legendrian link has $g(\La)=1$ and two components $\pi_0(\La)=2$. Intuitively, it is the geometric link corresponding to the $\wt A_{1,1}$ cluster algebra.\\
	
	\item[(iii)] According to (ii) above, the missing ingredient for Conjecture \ref{conj:ADE} is showing that (1), (2) and (3) hold. For the $A_n$-case (1), it is known that there are {\it at least} the stated Catalan number worth of exact Lagrangian fillings, distinct up to Hamiltonian isotopy. This was originally proven by Y. Pan \cite{PanCatalanFillings} and subsequently understood in \cite{STWZ,TreumannZaslow} from the perspective of microlocal sheaf theory. It remains to show that any exact Lagrangian filling of $\La(A_n)$ is Hamiltonian isotopic to one of those; the first unsolved case is the Hopf link $\La(A_1)$ having exactly two embedded exact Lagrangian fillings.\footnote{In particular, this would show that the {\it two} possible Polterovich surgeries \cite{Polterovich91} of a 2-dimensional Lagrangian node are the only two exact Lagrangian cylinders near the node, up to Hamiltonian isotopy.} For the $\La(D_n),\La(E_6),\La(E_7)$ and $\La(E_8)$ cases in Conjecture \ref{conj:ADE}, one needs to first find the corresponding number of distinct Lagrangian fillings, and then show these are all. The construction part should be relatively accessible, in the spirit of either \cite{CasalsZaslow,PanCatalanFillings,STWZ}, and it is reasonable to suspect that these many fillings can be distinguished using either augmentations or microlocal monodromies.\footnote{Showing these exhaust all fillings, up to Hamiltonian isotopy, is another matter, possibly much more challenging.}\\
	
	\item[(iv)] The numbers appearing in Conjecture \ref{conj:ADE}.(i)-(iii) are the number of cluster seeds for the corresponding cluster algebra. Precisely, consider a root system of Cartan-Killing type $X_n$, $e_1,\ldots,e_n$ its exponents and $h$ the Coxeter number. Then the numbers in Conjecture \ref{conj:ADE} are $N(X_n)=\prod_{i=1}^n(e_i+h+1)(e_i+1)^{-1}$ for $X_n=A_n,D_n,E_6,E_7,E_8$.
\end{itemize}

The brick graph of a positive braid is defined in \cite{BLL,Rudolph92_AnnuliIV}, it can be enhanced to a quiver, which we call the brick quiver, following the algorithm in \cite[Section 3.1]{ShenWeng19} or \cite[Section 4.2]{GSW20_1}, which itself generalizes the wiring diagram construction in \cite{BFZ05,FockGoncharov06_Amalgamation}.

\begin{remark} The hypothesis of the mutable part of its brick quiver being connected is necessary. We could otherwise add a meridian to any positive braid, which would create a disconnected quiver; the resulting cluster algebra would be a product with $A_1$, which preserves being of finite type. It stands to reason that adding a meridian to a Legendrian link $\La$ would yield a Legendrian link $\La\cup\mu$ with exactly {\it twice} as many Lagrangian fillings. It is clear that there are at least twice as many Lagrangian fillings for $\La\cup\mu$, as there are two distinct Lagrangian cobordisms from $\La$ to $\La\cup\mu$. The simplest case is $\La=\La_0$ the standard Legendrian unknot and $\La\cup\mu\cong\La(A_1)$ the Hopf link, which should have $2=2\cdot1$ Lagrangian fillings, in accordance with Conjecture \ref{conj:ADE}. The next case would be $\La=\La(A_1)$, so that $\La(A_1)\cup\mu\cong\La(D_2)$, in line with $\La(D_2)$ conjecturally having $4=2\cdot 2$ Lagrangian fillings.\hfill$\Box$
\end{remark}

Note that the article \cite{CasalsNg20} has provided the first examples of Legendrian links $\La\sse(\S^3,\xi_\st)$ which are {\it not} rainbow closures of positive braids and yet they admit infinitely many Lagrangian fillings, up to Hamiltonian isotopy. These Legendrian links have components which are stabilized, not max-tb, and thus they cannot be rainbow closures of any positive braid. It would be interesting to extend Conjecture \ref{conj:ADE} to a larger class of links, possibly including $(-1)$-framed closures of positive braids, as studied in \cite{CasalsNg20}.\\

\begin{remark}
To the author's knowledge, \cite{EliashbergPolterovich96,PanCatalanFillings}, Theorem \ref{thm:main}, and the recent \cite{CasalsHonghao,CasalsZaslow,CasalsNg20,GSW20_1,GSW20_2}, constitute the current evidence towards Conjecture \ref{conj:ADE}. That said, parts of Conjecture \ref{conj:ADE} might have appeared in the symplectic folklore in one form or another. The advent of Symplectic Field Theory led to the mantra of ``pseudoholomorphic curves or nothing''\footnote{That is, if pseudoholomorphic invariants cannot distinguish two objects, they must be equal.}, the subsequent arrival of microlocal sheaf theory to symplectic topology led to ``sheaves or nothing''. In the current zeitgeist, cluster algebras provide a new algebraic invariant that one might hope to be complete.\footnote{As with the previous two cases, there is no particularly hard evidence for ``cluster algebras or nothing''.} In this sense, I would like to mention Y. Eliashberg, D. Treumann, H. Gao, D. Weng and L. Shen as some of the colleagues which might have also discussed or hinted towards parts of Conjecture \ref{conj:ADE}.\hfill$\Box$
\end{remark}

Finally, an ADE-classification is often part of a larger classification\footnote{The larger classification is an ABCDEFG-classification, which admittedly does not roll off the tongue.}, involving a few additional families. For instance, simple Lie algebras are classified by connected Dynkin diagrams, which are $A_n,D_n,E_6,E_7,E_8$, known as the simply laced Lie algebras, and $B_n,C_n,F_4$ and $G_2$. These latter cases, $B_n,C_n,F_4$ and $G_2$, are interesting on their own right. For instance, simple singularities are classified according to $A_n,D_n,E_6,E_7,E_8$, and $B_n,C_n,F_4$ then arise in the classification of simple {\it boundary} singularities \cite[Chapter 17.4]{AGZV1}, as shown in \cite[Chapter 5.2]{AGZV2}. (See also D. Bennequin's \cite[Section 8]{Bennequin86} and \cite{Arnold78}.) In general, the tenet is that $B_n,C_n,F_4$ and $G_2$ arise when classifying the same objects as in the ADE-classification {\it with the additional} data of a symmetry.\footnote{The study of boundary singularities can be understood as the study of singularities taking into account a certain $\Z_2$-symmetry.} This a perspective (and technique) called {\it folding}, ubiquitous in the study of $B_n,C_n,F_4,G_2$, which is developed in \cite[Section 2.4]{FominZelevinsky03} for the case of cluster algebras.

Let us consider a Legendrian $\La\sse(\R^3,\xi_\st)$, a Lagrangian filling $L\sse(\R^4,\la_\st)$, $\dd L=\La$, and a finite group $G$ acting faithfully on $(\R^4,\la_\st)$ by exact symplectomorphisms, inducing an action on the boundary piece $(\R^3,\xi_\st)$ by contactomorphisms. For instance, $s:\R^4\lr\R^4$, $s(x,y,z,w)=(-x,-y,z,w)$ is an involutive symplectomorphism which restricts to the contactomorphism $(x,y,z)\mapsto(-x,-y,z)$ on its boundary piece $(\R^3,\ker\{dz-ydx\})$. Let us define an exact Lagrangian $G$-filling of $\La$ to be an exact Lagrangian filling $L$ of $\La$ such that $G(L)=L$ and $G(\La)=\La$ setwise. Also, by definition, we say $\La\sse(\R^3,\xi_\st)$ admits a $G$-symmetry if there exists a faithful action of $G$ by contactomorphisms on $(\R^3,\xi_\st)$ such that $G(\La)=\La$ setwise. Examples of such symmetries can be readily drawn in the front projection, as shown in Figure \ref{fig:BoundarySing} for $\La(A_9),\La(D_8),\La(E_6)$ and $\La(D_4)$. Following the tenet above, the following classification might be plausible:

\begin{conj}[BCFG Classification of Lagrangian Fillings]\label{conj:BCFG} Let $\La(\beta)\sse(\S^3,\xi_\st)$ the Legendrian rainbow closure of a positive braid $\beta$:
	
	\begin{itemize}
		\item[1.] $(B_n)$ If $\La(\beta)=\La(A_{2n-1})$, the $\Z_2$-symmetry $(x,z)\lr(-x,z)$ for the front depicted in Figure \ref{fig:BoundarySing} lifts to a $\Z_2$-symmetry of $\La(A_{2n-1})$. Then $\La(A_{2n-1})$ has precisely ${2n\choose n}$ exact Lagrangian $\Z_2$-fillings.\\
		
		\item[2.] $(C_n)$ If $\La(\beta)=\La(D_{n+1})$, the $\Z_2$-symmetry $(x,z)\lr(-x,z)$ for the front depicted in Figure \ref{fig:BoundarySing} lifts to a $\Z_2$-symmetry of $\La(D_{n+1})$. Then $\La(D_{n+1})$ has precisely ${2n\choose n}$ exact Lagrangian $\Z_2$-fillings.\\
		
		\item[3.] $(F_4)$ If $\La(\beta)=\La(E_6)$, the $\Z_2$-symmetry $(x,z)\lr(-x,z)$ in the front depicted in Figure \ref{fig:BoundarySing} lifts to a $\Z_2$-symmetry of $\La(E_6)$. Then $\La(E_6)$ has precisely $105$ exact Lagrangian $\Z_2$-fillings.\\
		
		\item[4.] $(G_2)$ If $\La(\beta)=\La(D_4)$, the $\Z_3$-symmetry in the front depicted in Figure \ref{fig:BoundarySing} lifts to a $\Z_3$-symmetry of $\La(D^4)$. Then $\La(D_4)$ has precisely $8$ exact Lagrangian $\Z_3$-fillings.\\
	\end{itemize}
\end{conj}

\begin{center}
	\begin{figure}[h!]
		\centering
		\includegraphics[scale=0.8]{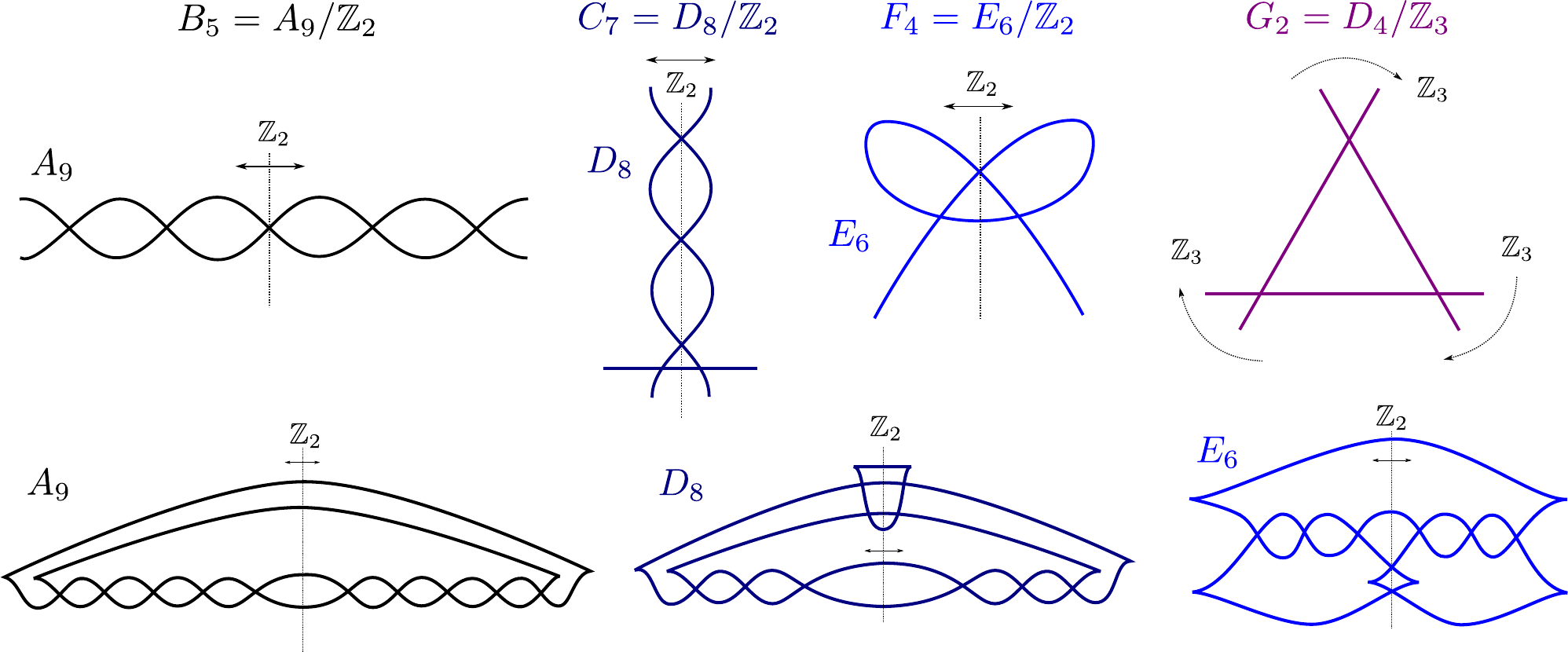}
		\caption{Legendrian fronts for $\La(A_{2n-1}),\La(D_{n+1}),\La(E_6),\La(D_4)$ with $G$-symmetries, $G=\Z_2,\Z_3$. The upper row exhibits these symmetric fronts as divides of the associated singularities, and the lower row depicts them in the standard front projection $(x,y,z)\mapsto(x,z)$ for a Darboux chart $(\R^3,\xi_\st)$.}
		\label{fig:BoundarySing}
	\end{figure}
\end{center}

For the $G_2$-case in Conjecture \ref{conj:BCFG}.(4), it might be helpful to notice that the $D_4$-singularity is topologically equivalent to $f(x,y)=x^3+y^3$. The $\Z_3$-symmetry cyclically interchanges the three linear branches of this singularity. In particular, we can draw a front for the Legendrian $\La(D_4)$ as the $(3,3)$-torus link, the rainbow closure of $\beta=(\sigma_1\sigma_2)^6$.\footnote{The $\Z_3$-action should coincide with the loop $\Xi_1\circ(\delta^{-1}\circ\Xi_1\circ\delta)$ from \cite[Section 2]{CasalsHonghao}.}

For the $B_n$-case in Conjecture \ref{conj:BCFG}.(1), the construction of ${2n \choose n}$ distinct Lagrangian $\Z_2$-fillings likely follows from adapting \cite{PanCatalanFillings}. Indeed, in the $\Z_2$-invariant front for $\La(A_{2n-1})$, as depicted in Figure \ref{fig:BoundarySing}, there are $n$ crossing to the left, equivalently right, of the $\Z_2$-symmetry axis. We can construct a $\Z_2$-filling of $\La(A_{2n-1})$ by opening those $n$ crossings in any order, with the rule that we simultaneously open the corresponding $\Z_2$-symmetric crossing.\footnote{The naive count of $312$-pattern avoiding permutations from \cite{EkholmHondaKalman16,PanCatalanFillings} would indicate that there are $\frac{1}{n}{2n\choose n}$ such Lagrangian $\Z_2$-fillings, instead of ${2n\choose n}$. Thus, should Conjecture \ref{conj:BCFG} hold, there must be an additional rule for $\Z_2$-fillings (not just those in \cite[Lemma 3.10]{PanCatalanFillings}), possibly related to the fact that the crossing closest to the $\Z_2$-axis is different from the rest.} Should one distinguish these $\Z_2$-fillings via their augmentations, as in \cite{PanCatalanFillings}, an appropriate $G$-equivariant Floer theoretic invariant (e.g. $G$-equivariant DGA and its augmentations) needs to be defined. The perspective of microlocal sheaves \cite{TreumannZaslow} yields combinatorics closer to those of triangulations \cite[Section 12.1]{FominZelevinsky03_Finite}, modeling $A_n$-cluster algebras, and thus might provide a simpler route to distinguish these fillings. In either case, Conjecture \ref{conj:BCFG} calls for a $G$-equivariant theory of invariants for Legendrian submanifolds of contact manifolds.

\subsection{Some Questions} We finalize this section with a series of problems on Weinstein 4-manifolds and their Lagrangian skeleta. To my knowledge, there are several unanswered questions at this stage, including checkable characterizations of Weinstein 4-manifolds of the form $W(\La_f)$, where $\La_f$ is the Legendrian link of an isolated plane curve singularity. Here are some interesting, yet hopefully reasonable, problems:

\color{coolblack}{\it Problem 1}. \color{black} Find a characterization of Legendrian links $\La\sse(\S^3,\xi_\st)$ for which $(\C^2,\La)$, or $W(\La)$, admits a Cal-skeleton. (Ideally, a verifiable characterization.)

\color{coolblack}{\it Problem 2}. \color{black} Find necessary and sufficient conditions for a Lagrangian skeleton $\L\sse(W,\la)$ to guarantee that the Stein manifold $(W,\la)$ is an affine algebraic manifold. Similarly, characterize Legendrian links $\La\sse(\S^3,\xi_\st)$ such that $W(\La)$ is an affine algebraic variety.

Note that the standard Legendrian unknot $\La_0\cong\La(A_0)\sse(\S^3,\xi_\st)$ and the max-tb Hopf link $\La(A_1)\sse(\S^3,\xi_\st)$ yield {\it affine} Weinstein manifolds, as we have
$$W(\La_0)\cong\{(x,y,z)\in\C^3:x^2+y^2+z^2=1\},\quad W(\La(A_1))\cong\{(x,y,z)\in\C^3:x^3+y^2+z^2=1\}.$$
By \cite[Section 4.1]{CasalsMurphy19}, the trefoil $\La(A_2)$ is also an example of such a Legendrian link, as
$$W(\La(A_2))\cong\{(x,y,z)\in\C^3:xyz+x+z+1=0\}.$$
Heuristic computations indicate that $\La(A_3)$ and $\La(D_4)$ also have this property. See \cite{McLean12,McLean18} for a source of necessary conditions, and \cite{NRTZ14} for (topological) skeleta of affine hypersurfaces.

\color{coolblack}{\it Problem 3}. \color{black} Find necessary and sufficient conditions for a Lagrangian skeleton\footnote{Not closed in this case.} $\L\sse(W,\la)$ to guarantee that the Stein manifold $(W,\la)$ is flexible.\footnote{See \cite{CieliebakEliashberg12} for flexible Weinstein manifolds. In the 4-dimensional case above, we might just define flexible as being of the form $W=W(\La)$ where $\La$ is a stabilized knot.} (Again, a verifiable characterization.) Similarly, characterize $\La\sse(\S^3,\xi_\st)$ such that $W(\La)$ is flexible.

Note that affine manifolds $W\sse\C^N$ might be flexible \cite[Theorem 1.1]{CasalsMurphy19}. In particular, it could be fruitful to compare Lagrangian skeleta of $X_m=\{(x,y,z)\in\C^3:x^my+z^2=1\}$ for $m=1$ and $m\geq2$, e.g. the ones provided in \cite{NRTZ14}.

\color{coolblack}{\it Problem 4}. \color{black} Suppose that a Weinstein 4-manifold $W=W(\La)$ is obtained as a Lagrangian 2-handle attachment to $(\D^4,\omega_\st)$. Given a Cal-skeleton $\L\sse(W,\la)$, devise an algorithm to find one such possible Legendrian $\La\sse(\dd\D^4,\xi_\st)$.

\color{coolblack}{\it Problem 5}. \color{black} Let $L\sse(W,\la)$ be a closed exact Lagrangian surface. Study whether there exists a Cal-skeleton $\L\sse(W,\la)$ such that $L\sse\L$. In addition, study whether there exists a Legendrian handlebody $\La\sse(\#^k\S^1\times\S^2,\xi_\st)$, so that $W=W(\La)$, and $L$ is obtained by capping a Lagrangian filling of a Legendrian sublink of $\La$.

See \cite{Verine20} for an interesting construction in the case of Bohr-Sommerfeld Lagrangian submanifolds and see \cite{EGL20} for a general discussion on regular Lagrangians. The nearby Lagrangian conjecture holds for $W=T^*\S^2,T^*\mathbb{T}^2$, thus the answer is affirmative in these cases.

\color{coolblack}{\it Problem 6}. \color{black} Characterize which cluster algebras $A$ can arise as the ring of functions of the augmentation stack of a Legendrian link $\La\sse(\S^3,\xi)$.

By using double-wiring diagrams \cite{BFZ05}, (generalized) double Bruhat cells satisfy this property \cite{ShenWeng19}. It is proven in \cite{CasalsNg20,GSW20_1} that the cluster algebras $A(\wt D_n)$ of affine $D_n$-type have this property. Heuristic computations indicate that the affine types $\wt A_{p,q}$ also verify this \cite{CasalsNg20}. It might be reasonable to conjecture that cluster algebras of surface type all have this property.

\color{coolblack}{\it Problem 7}. \color{black} Let $a_3(\La)$ be the number of $A_3$-arboreal singularities of a Cal-skeleton $\L\sse(W,\la)$. Find the number $a_3(W):=\min_{\L\sse W}a_3(\L)$, where $\L\sse W$ runs amongst all possible Cal-skeleta. In particular, characterize Weinstein 4-manifolds $(W,\la)$ with $a_3(W)=0$.

\color{coolblack}{\it Problem 8}. \color{black} Develop a combinatorial theory of symplectomorphisms in $\Symp(W,d\la)$ in terms of Cal-skeleta $\L\sse(W,\la)$.

This is being developed in the case $\dim(W)=2$ by using A'Campo's t\^{e}te-\`a-t\^{e}te twists \cite[Section 3]{ACampo18}, see also \cite[Section 5]{AFPP17}. A (symplectic) mapping class in $\Symp(W,d\la)$ is a composition of Dehn twists in this 2-dimensional case. This is no longer the case in $\dim(W)=4$, e.g. due to the existence of Biran-Giroux's fibered Dehn twists, confer \cite[Section 3]{Uljarevic17} and \cite[Section 2]{WehrheimWoodward16}. Note that $\pi_0(\Symp(W))$ might be infinite even if $W$ contains no exact Lagrangian 2-spheres \cite{CasalsHonghao}.

\color{coolblack}{\it Problem 9}. \color{black} Compare Cal-skeleta $\L_1\sse(W_1,\la_1)$, $\L_2\sse(W_2,\la_2)$ for exotic Stein pairs $W_1,W_2$. That is, $W_1$ is homeomorphic to $W_2$, but not diffeomorphic. In particular, investigate {\it skeletal corks}: combinatorial modifications on a Cal-skeleton that can produce exotic Stein pairs.

In \cite{Naoe18}, H. Naoe uses Bing's house \cite{Bing64} to study some such corks.

\color{coolblack}{\it Problem 10}. \color{black} Find a contact analogue of Turaev's Shadow formula\footnote{This expresses the $SU(2)$-Reshetikhin-Turaev-Witten quantum invariant of a 3-manifold in terms of a shadow as a (colored) multiplicative Euler characteristic.} \cite[Chapter 10]{Turaev94} for the contact 3-dimensional boundary in terms of the combinatorics of a Cal-skeleton $\L\sse(W,\la)$. That is, find a {\it contact} invariant\footnote{E.g. it would be interesting to describe the Ozsvath-Szabo contact class in Heegaard Floer homology, or M. Hutching's contact class in Embedded Contact homology, in terms of $\L$ as well.} of $(\dd W,\la|_{\dd W})$ which can be computed in terms of the combinatorics of $\L\sse(W,\la)$.


\bibliographystyle{plain}
\bibliography{LagrSkel_Casals}

\end{document}